\documentclass[a4paper]{article}
\usepackage[T1]{fontenc}
\usepackage[utf8x]{inputenc}

\usepackage{pgfplots}
\pgfplotsset{compat=newest}
\usetikzlibrary{arrows,patterns,external}
\usepgfplotslibrary{fillbetween,groupplots}

\usepackage{tabularx}
\usepackage[english]{babel}
\usepackage{amsfonts,amssymb,amsmath,amsthm,mathrsfs}
\usepackage{graphicx}
\usepackage{hyperref}
\usepackage[margin=2cm]{geometry}

\usepackage{newpxtext,newpxmath}
\usepackage{enumerate}
\usepackage{xspace}
\usepackage{stackrel}

\numberwithin{equation}{section}
\numberwithin{figure}{section}
\newtheorem{theorem}{Theorem}[section]
\newtheorem{prop}[theorem]{Proposition}
\newtheorem{lemma}[theorem]{Lemma}
\theoremstyle{definition}
\newtheorem{remark}[theorem]{Remark}

\newcommand{\dsp}[1]{\displaystyle{#1}}
\newcommand{\txt}[1]{\textstyle{#1}}
\newcommand{\nn}{\nonumber}
\newcommand{\beq}{\begin{equation}}
\newcommand{\eeq}{\end{equation}}
\newcommand{\vect}[1]{\begin{pmatrix}#1\end{pmatrix}}

\newcommand{\mb}[1]{\ensuremath{\mathbb{#1}}}

\newcommand{\RR}{\mb{R}}

\newcommand{\dint}{\mathrm{d}}
\newcommand{\dds}{\frac{d}{ds}}

\newcommand{\A}{\mathscr{A}}
\renewcommand{\P}{\mathscr{P}}
\newcommand{\veps}{\mu}
\def\intfull#1{\int_0^{2\pi}#1\,\dint s}
\newcommand{\torusL}{\mathbb{T}_{L}}
\newcommand{\torustwopi}{\mathbb{T}_{2\pi}}

\title{{\bf Bifurcation of elastic curves with}\\
{\bf modulated stiffness}}
\author{Katharina Brazda$^1$,
Gaspard Jankowiak$^2$,
Christian Schmeiser$^1$,
Ulisse Stefanelli$^{1,3,4}$
\vspace{1.5em}\\
\footnotesize{$^1$University of Vienna, Faculty of Mathematics, Oskar-Morgenstern-Platz 1, 1090 Wien, Austria}\\
\footnotesize{$^2$Radon Institute for Applied and Computational Mathematics, Altenbergerstr. 69, 4040, Linz, Austria}\\
\footnotesize{$^3$Vienna Research Platform on Accelerating
Photoreaction Discovery, University of Vienna, W\"ahringerstra\ss e 17, 1090 Wien, Austria}\\
\footnotesize{$^4$Istituto di Matematica Applicata e Tecnologie Informatiche \emph{E. Magenes}, via Ferrata 1, I-27100 Pavia, Italy}}

\date{\today}

\begin{document}

\maketitle

\begin{abstract} We investigate the equilibrium configurations of
  closed planar elastic curves of fixed length, whose stiffness,
  also known as the bending rigidity, depends on an additional density variable. The
  underlying variational model relies on the minimization of a bending
  energy with respect to shape and density and can be considered as a
  one-dimensional analogue of the Canham-Helfrich model for
  heterogeneous biological membranes. We present a generalized
  Euler-Bernoulli elastica functional featuring a density-dependent
  stiffness coefficient. In order to treat the inherent nonconvexity
  of the problem we introduce an additional length scale in the
  model by means of a density gradient term. We derive
  the system of Euler-Lagrange equations and study the bifurcation
  structure of solutions with respect to the model parameters.
  Both analytical and numerical results are presented.
\end{abstract}

\par\emph{MSC 2020: }
%
  35J20,   
%
  35B38,   
  35B32,   
  35B36,   
%
%
  74G65,   
%
  74K10    


\setcounter{tocdepth}{2}
\tableofcontents

\section{Introduction}
We investigate the equilibrium configurations of elastic curves featuring an additional scalar density variable
which influences the bending rigidity. Our interest is
motivated by the variational modelization of the 
 \emph{shapes of biological membranes}, originally proposed by \textsc{Canham} \cite{canham_minimum_1970} and \textsc{Helfrich} \cite{helfrich_elastic_1973} to explain the characteristic biconcave shape of a human red blood cell. According to this model, the equilibrium membrane shape $\Sigma$ minimizes the bending energy
\begin{equation*}
E_{\mathrm{CH}}(\Sigma)=\int_\Sigma \left(\frac{\beta}{2}\,(H-H_0)^2+\gamma\, K\right)\dint S
\end{equation*}
under suitable constraints on membrane area and enclosed volume.
Here, $\Sigma$ is a smooth closed surface embedded in $\RR^3$, $H$
and $K$ are the mean and the Gauss curvature of
$\Sigma$, respectively, 
and the material parameters comprise the stiffnesses (bending
rigidities) $\beta>0$, $\gamma<0$ as well as the spontaneous curvature
$H_0\in\RR$. The material parameters of heterogeneous biomembranes
are assumed to  depend on the variable membrane composition,
which is described by a scalar function $\rho\colon
\Sigma\to\RR$ which we interpret as a density of fixed total
mass. On the other hand, the geometry of the membrane influences
the distribution of the density $\rho$, which originates a 
  \emph{coupling effect between curvature and composition}.
  Indeed, the energy for heterogeneous biomembranes has to be
  minimized with respect to both membrane geometry $\Sigma$ and
  composition $\rho$ simultaneously. Configurations featuring this coupling have been experimentally observed for example by \textsc{Baumgart, Hess, \& Webb} \cite{baumgart_imaging_2003} in case of giant unilamellar vesicles. Furthermore, the coupling effect also plays an essential role in the dynamic morphology changes of cells, where special curved membrane proteins are involved, cf.\ \textsc{McMahon \& Gallop}~\cite{mcmahon_membrane_2005}.

Results on the mathematical analysis of the variational problem for heterogeneous biomembranes have been obtained 
by \textsc{Choksi, Morandotti, \& Veneroni} \cite{choksi_global_2013} and
\textsc{Helmers} \cite{helmers_convergence_2015}, who proved the
existence of multiphase minimizers in the axisymmetric
regime. By dropping the symmetry restriction, existence of
multiphase minimizers has been  recently obtained by
\cite{brazda_existence_2020} in the weak setting of {\it varifolds}.
For a collection of recent results on both single- and multiphase Canham-Helfrich models the reader is referred to \cite{barrett_ea_2018_twophase, brazda_existence_2020, deckelnick_ea_2021_bvprevolution, eichmann_2020_lsc, elliott_ea_2020_domain, lussardi_2019_mesoscopic, mondino_existence_2020, peletier_ea_2009_bilayers, wojtowytsch_2019_constrained}.

To the best of our knowledge, proving existence of
minimizers for membranes featuring  continuous phase densities and
general material parameter models is an open problem. We move a
first step in this direction in the present paper, by focusing on the lower-dimensional
setting of curves instead.  
A classical elastic curve in the plane, $\gamma\colon [0,L]\to\RR^2$,
minimizes the {\sc Euler-Bernoulli} elastic bending energy (also known as the \textsc{Willmore} energy)
\[
E(\gamma)=\frac{1}{2}\int_\gamma\kappa^2\,\dint s,
\]
where $\kappa$ is the scalar curvature of $\gamma$. The stationary points are called \emph{elasticae} and can be analytically described in terms of elliptic functions. As was already clear to Euler, the only closed elasticae of fixed length in the plane are the circle and Bernoulli's Figure-8 curve, the single covered circle being the unique global minimizer of $E$, see for example \textsc{Truesdell} \cite{truesdell_elasticity_1983} and \textsc{Langer \& Singer} \cite{langer_curve_1985}.

We now modify the setting by taking the 
additional scalar density $\rho$ into the picture. The density $\rho$ modulates the elastic behavior of the curve. For this purpose we consider the following elastic bending energy with \emph{density-modulated stiffness},
\[
E_0(\rho,\gamma)=\frac{1}{2}\int_\gamma\beta(\rho)\,\kappa^2\dint s.
\]
Our interest lies on the effects of the variable stiffness $\beta$
and we dispense with the 
spontaneous curvature $H_0$, for simplicity.
In order to take into account the coupling between shape and
composition, we have to minimize $E_0$ with respect to both $\gamma$
and $\rho$. Admissible curves $\gamma$ are asked to be
planar, regular, $C^1$-closed, and have fixed length $L$, whereas
admissible densities $\rho$ are required to have fixed mass
$\int_\gamma\rho\,\dint s=M$.

The application of the Direct Method for the minimization of $E_0$
calls for checking lower semicontinuity with respect to weak topologies,
which in turn asks for the convexity of the
integrand of $E_0$. Yet, if such convexity is imposed, only
the trivial minimizer exists, namely the constant density $\rho_0=M/L$
on a circle with curvature $\kappa_0=2\pi/L$. This however is insufficient
for describing 
the rich geometric morphologies that can be observed in
biological membranes.

In the following, we will therefore not assume convexity
of 
the integrand of $E_0$. This lack of convexity may however lead to nonexistence of
minimizers,  see Section \ref{sec:existence} below.  We are hence
forced to consider a
regularized 
energy $E_\mu$, featuring an additional length scale in terms of
a gradient term in $\rho$, namely,
\beq\label{eq:Ereg}
E_\veps(\rho,\gamma)=\frac{1}{2}\int_\gamma\left(\beta(\rho)\,\kappa^2+\veps\,\dot\rho^2\right)\dint s,
\eeq
where $\dot \rho := \frac{\mathrm{d}}{\mathrm{d} s} \rho$.
The parameter $\mu$ may be physically interpreted as the diffusivity of
the density, cf.\ \eqref{eq:Ega}.
For $\mu$ large, the
only minimizer is the trivial one, see Proposition \ref{thm:trivial_unique_large_epsilon}. By lowering $\mu$ one observes the
onset of bifurcations from the trivial state. The main focus of this paper is the
rigorous  bifurcation analysis in terms of $\mu$. We
analytically classify the bifurcation behavior of solutions
of the Euler-Lagrange equations of 
$E_\veps$. Moreover, we provide an exhaustive suite of numerical experiments, illustrating the distinguished patterning of minimizers of $E_\mu$, depending on $\mu$.

A variational model for planar elastic curves with density has also been studied by \textsc{Helmers} \cite{helmers_snapping_2011}. He focused on the effect of spontaneous curvature and established a $\Gamma$-convergence result to the sharp interface limit. Let us mention also the recent work by \textsc{Palmer \& P\'ampano} \cite{palmer_anisotropic_2020}, who presented analysis and numerics for the shapes of elastic rods with anisotropic bending energies.

We conclude this introduction by presenting the outline of the paper. In Section \ref{sec:background}, we briefly describe the mathematical setting and explain our notation. Section \ref{sec:existence} is devoted to the justification of our model by existence and nonexistence results for minimizers. In Section \ref{sec:bifurcation} we analytically discuss the local bifurcation structure of solutions to the associated Euler-Lagrange equations. Numerical results for the bifurcation branches as well as for the configurations of the curves are presented in Section \ref{sec:numerics}. Finally, Section \ref{sec:conclusion} summarizes our findings.


\section{Mathematical setting}\label{sec:background}

We devote this section to make the mathematical setting precise and fix notation.

\subsection{Notation and preliminaries on curves}
We collect some basic information on curves \cite{do_carmo_differential_1976}. In the following, we will consider \emph{closed planar curves} $\gamma\in H^2(\torusL)^2$, where $\torusL := \mathbb{R}/L\mathbb{Z}$ is the one-dimensional torus with period $L>0$. The fact that $H^2(\torusL)\subset C^1(\torusL)$ ensures that $\gamma\colon[0,L]\to\mathbb{R}^2$ represents a $C^1$-closed curve and $\gamma(0)=\gamma(L)$ and $\dot\gamma(0)=\dot\gamma(L)$.
We systematically assume $\gamma$ to be parametrized by arc-length $s$,  namely, $|\dot\gamma| =1$.
This induces that $\ddot\gamma\in L^2(\torusL)^2$ is orthogonal to $\dot\gamma$. The normal vector $n$ to the curve is defined pointwise by counterclockwise rotating $\dot\gamma$ by  $\pi/2$. That is, by denoting $\gamma(s) = (x(s), y(s))$, $n(s)=\dot\gamma(s)^\bot:=(-\dot y(s), \dot x(s))$.
The rate of change of $\dot\gamma$ in direction $n$ is measured by the scalar \emph{curvature}
$\kappa=n\cdot\ddot\gamma=\det(\dot\gamma,\ddot\gamma)\in L^2(\torusL)$ of the curve, so that
$\ddot\gamma=(n\cdot\ddot\gamma)\,n=\kappa\,n$.

The \emph{inclination angle} $\theta\in L^2(\torusL)$ is the angle between the $x$-axis and the tangent $\dot\gamma$, that is $\dot\gamma=(\dot x,\dot y) = (\cos\theta,\sin\theta)$. Note that in even for smooth $\gamma$, $\theta$ is discontinuous on $\torusL$. However, 
the map $s\mapsto \theta(s) - \frac{2\pi}{L}\, I\, s$ is an element of $H^1(\torusL)$,
where the \emph{rotation index} of the curve $I\in \mathbb{Z}$ counts the number of complete turns of $\dot \gamma$ according to the standard orientation, see below.
The curvature function $\kappa\in L^2(\torusL)$ uniquely determines the curve $\gamma\in H^2(\torusL)^2$  up to translations and rotations in $\RR^2$ \cite[Section 1-5, pp.\ 19, 24, and Section 1-7, p.\ 36]{do_carmo_differential_1976}. In particular, if $|\dot\gamma|=1$, then
\beq
\label{eq:thetakappa}
\kappa=\dot\theta, \qquad \theta(s')=\theta(0)+\int_{0}^{s'}\kappa(s'')\,\dint s'',
\quad\text{and}\quad
\gamma(s)=\vect{x(s)\\ y(s)}
=\vect{x(0)\\ y(0)}+\int_0^s\vect{\cos\theta(s') \\ \sin\theta(s')}\dint s'.
\eeq
Identifying all curves whose images only differ by isometries in $\RR^2$, one may adapt the coordinate system to $x(0)=y(0)=\theta(0)=0$, corresponding indeed to the choice $\gamma(0)=(0,0)$ and $\dot\gamma(0)=(1,0)$.
A curve $\gamma\in H^2(\torusL)^2$
parametrized by arc-length 
satisfies
the following identities:
\begin{align*}
&0=\gamma(L)-\gamma(0)=\int_0^L\dot \gamma(s)\dint s=\int_0^L\vect{\cos\theta(s) \\ \sin\theta(s)}\dint s
=\int_0^L\vect{\cos\left(\theta(0)+\int_{0}^s\kappa(t)\,\dint t\right)
  \\ \sin\left(\theta(0)+\int_{0}^s\kappa(t)\,\dint t\right)}\dint s\,, \\
&{0=\dot\gamma(L)-\dot\gamma(0)=(\cos\theta(L) - \cos\theta(0),
  \sin\theta(L)-\sin\theta(0))}\,.
\end{align*}
The latter is equivalent to $\theta(L)-\theta(0)=\int_0^L\kappa(s)\,\dint s=2\pi \, I$.
A curve $\gamma\colon [0,L]\to\RR^2$ is called \emph{simple} if it is an injective map and \emph{regular}, if it is $C^1$ and $\dot\gamma(t)\neq 0$ for all $t\in[0,L]$.
By the {Theorem of Turning Tangents} \cite[Section 5-6, Theorem 2, p.\ 396]{do_carmo_differential_1976}, a simple $C^1$-closed regular planar positively oriented $C^1$ curve 
has rotation index $I=1$.
This allows us to represent a simple $C^1$-closed curve $\gamma\in
H^2(\torusL)^2$ parametrized by arc-length by its inclination angle
$\theta$, granted that $\theta - \frac{2\pi}{L}\, I\, s\in H^1(\torusL)$ and
\[
{\theta(0)=0, \ \ \theta(L)=2\pi
\quad\text{and}\quad
\int_0^L\vect{\cos\theta(s) \\ \sin\theta(s)}\dint s=0},
\]
or by its curvature $\kappa\in L^2(\torusL)$, additionally satisfying
\[
\dsp{\int_0^L\kappa(s)\,\dint s=2\pi}\quad
\text{and}
\quad \dsp{\int_0^L\vect{\cos\left(\int_{0}^s\kappa(t)\,\dint t\right)
\\
\sin\left(\int_{0}^s\kappa(t)\,\dint t\right)}\dint s=0}.
\]
Eventually, note that by requiring a planar curve to be closed restricts the possible curvature functions.
According to the {Four Vertex Theorem}
\cite[Section 1-7, Theorem 2, p.\ 37]{do_carmo_differential_1976},
a smooth simple closed regular planar curve has either constant curvature (i.e.\ is a circle) or the curvature function possesses at least four vertices, i.e.\ two local minima and two local maxima.
The converse statement is given in \cite{dahlberg_converse_2005}: every continuous function which either is a nonzero constant or has at least four vertices is the curvature of a simple closed regular planar curve.

\subsection{Elastic energies with modulated stiffness}\label{ssec:model}

We consider planar curves $\gamma\in H^2(\torusL)^2$ parametrized by arc-length. With no loss of generality, we will assume from now on the length $L$ of the curve to be $2\pi$. The scalar density field $\rho\colon[0,2\pi]\to {\mathbb  R}$ is considered to be a function of the arc-length of the curve.
Moreover, we are given a density-modulated stiffness
\begin{equation}
 \label{beta}
  \beta \in C^2(\RR) \ \ \text{with} \ \ \inf \beta =:\beta_m>0.
\end{equation}
In the following, we will assume \eqref{beta} to hold throughout, without explicit mention. Note however that some results in this section are valid under weaker conditions on $\beta$ as well.

\emph{Admissible curves} are defined as elements of the set
\begin{align*}
    \A & := \left\{\gamma\in H^2(\torustwopi)^2: \: |\dot\gamma|=1,\:\gamma(0)=\gamma(2\pi)=(0,0),\:\dot\gamma(0)=\dot\gamma(2\pi)=(1,0),\: \int_0^{2\pi}\det(\dot\gamma(s),\ddot\gamma(s))\,\dint s=2\pi \right\}.
\end{align*}
In particular, admissible curves are planar, arc-length parametrized, and $C^1$-closed.
Note that we are not enforcing injectivity of $\gamma$ (i.e.\ $\gamma$ being simple) and we just require the weaker condition   $I=1$. This simplifies our tractation, having no effect on the bifurcation result (Section \ref{sec:bifurcation}).

By the representation theorem for plane curves, any admissible curve $\gamma\in\A$ can be recovered from its inclination angle $\theta$ or its curvature $\kappa$. Correspondingly, we can equivalently indicate admissible curves as
\beq
\label{eq:Athetareg} 
\A=\left\{\theta\in L^2(\torustwopi): \theta -  s \in H^1(\torustwopi), \,
 \int_0^{2\pi}\vect{\cos\theta(s) \\
    \sin\theta(s)}\dint s=\vect{0\\ 0},\:\theta(0)=0
    \right\}
\eeq
or
\[ 
\A= \left\{\kappa\in L^2(\torustwopi): \: \int_0^{2\pi}\vect{\cos\left(\int_{0}^s\kappa(t)\dint t\right) \\ \sin\left(\int_{0}^s\kappa(t)\dint t\right)}\dint s=\vect{0\\ 0},\:\int_0^{2\pi}\kappa(s)\,\dint s=2\pi \right\}.
\]
The abuse of notation in defining the set $\A$ is motivated by the above-mentioned equivalence of the representations via $\gamma$, $\theta$, and $\kappa$, up to fixing
$\gamma(0)=(0,0)$ and $\dot \gamma(0)=(1,0)$ or $\theta(0)=0$.

\emph{Admissible densities} $\rho$ are asked to have fixed total mass. By possibly redefining $\beta$, one may assume such mass to be $2\pi$, which simplifies notation.
Given the parameter $\veps \in [0,\infty)$,
we define
\beq
\label{eq:Adensityreg} 
\P:=\left\{\rho\in L^1(\torustwopi): \: \veps\rho \in H^1(\torustwopi),\: \int_0^{2\pi}\rho(s)\,\dint s=2\pi
 \right\}.
\eeq
For the sake of simplicity, we do not restrict the values of $\rho$ to be nonnegative, which would however be sensible, for $\rho$ is interpreted as a density. Note however that this simplification has no effect on the bifurcation results, which are actually addressing a neighborhood of the trivial state only, where $\rho$ is constant and positive.

The \emph{elastic energy with modulated stiffness} is defined as
\beq
\label{eq:Ega}
E_\veps(\rho,\gamma):=\int_0^{2\pi} \left(\frac12 \beta(\rho) \ddot \gamma^2  + \frac{\veps}{2}
\dot \rho^2 \right) \dint s.
\eeq
Note that the energy $E_\veps$ can be equivalently rewritten as
$E_\veps(\rho,\gamma)=E_\veps(\rho,\theta)=E_\veps(\rho,\kappa)$, again by abusing notation.

We identify elastic curves with modulated stiffness as minimizers of $E_\veps$. In particular, we consider the following minimization problem
\beq
\label{eq:Egamin}
 \boxed{\min_{(\rho,\gamma)\in\P\times\A}E_\veps(\rho,\gamma).}
\eeq
In contrast to the classical Euler-Bernoulli model for elasticae \cite{langer_curve_1985, truesdell_elasticity_1983}, which is a purely geometric variational problem, here the density plays an active role in the selection of the optimal geometry.

\section{Existence and nonexistence}\label{sec:existence}

As mentioned in the Introduction, the minimization of $E_0$ turns out to be of limited interest. Indeed, if the integrand
\[
\Phi(\rho,\kappa) = \frac12 \beta(\rho) \kappa^2
\]
is strictly convex, problem  \eqref{eq:Egamin} for $\veps=0$ admits only the trivial solution
\beq \label{eq:trivial}
(\rho_0,\kappa_0):=(1,1).
\eeq
This can be directly checked  via Jensen's inequality by computing, for any $(\rho,\kappa)\in\P\times\A$,
\begin{align*}
  E_0(\rho,\kappa) = \int_0^{2\pi} \Phi(\rho,\kappa)
  \, \dint s \stackrel{\text{Jensen}}{\geq} 2\pi\,\Phi\left(
  \frac{1}{2\pi}\int_0^{2\pi} \rho \, \dint s,
  \frac{1}{2\pi}\int_0^{2\pi} \kappa \, \dint s \right) =
  2\pi\,\Phi(1,1) = E_0(\rho_0,\kappa_0)
\end{align*}
where the inequality is strict whenever $\rho$ or $\kappa$ are not constant, namely, whenever $(\rho,\kappa)\not =(\rho_0,\kappa_0)$. Let us mention that the integrand $\Phi$ is strictly
convex
if and only if
\beq
\label{eq:strictlyconvex}
\beta''>0\qquad \text{and}\qquad \beta''\beta>2(\beta')^2.
\eeq

In order to allow the complex geometrical patterning of biological shapes to possibly be described by the minimization problem \eqref{eq:Egamin}, one is hence forced to dispense of \eqref{eq:strictlyconvex}, for in that case the only minimizer of $E_0$ (and, a fortiori $E_\veps$) would be the trivial one $(\rho_0,\kappa_0)$. In the setting of our bifurcation results, our choices for $\beta$ will then fulfill
\begin{equation}\label{beta2}
\beta''(\rho)\leq 0\qquad \text{or}\qquad
\beta''(\rho)\beta(\rho)\leq 2(\beta'(\rho))^2\qquad \text{for some} \
\ \rho\geq 0,
\end{equation}
at least in a neighborhood of the trivial state $\rho_0$.

On the other hand, lacking convexity of the integrand $\Phi$, the energy $E_0$  fails to be weakly lower semicontinuous on $\P \times\A$, e.g.\ \cite[Thm.\ 5.14]{fonseca_modern_2007},  and existence of minimizers may genuinely fail. We collect a remark in this direction in the following.

\begin{prop}[No minimizers for $E_0$] \label{non} Assume that
  \begin{equation}\label{mon}
\beta(0)<\beta(\rho)\quad \forall\rho>0.
\end{equation} Then, the minimization problem
  \eqref{eq:Egamin} with $\veps=0$ admits no solution.
\end{prop}

Before moving to the proof, let us point out that condition \eqref{mon} implies in particular that $\Phi$ is not convex. Indeed, if $\Phi$ were convex, one could take any $\rho>0$ and $\lambda \in (0,1) $ and compute
\begin{align*}
\frac12 \beta(\rho)\kappa_0^2 &= \lim_{\lambda \to 1}
\Phi\big(\lambda (0,\kappa_0/\lambda) +
(1-\lambda) (\rho/(1-\lambda),0)\big)
\\
&\leq \lim_{\lambda \to 1} \Big( \lambda
  \Phi(0,\kappa_0/\lambda) + (1-\lambda) \Phi(\rho/(1-\lambda),0) \Big)
=\lim_{\lambda \to 1}  \lambda
  \Phi(0,\kappa_0/\lambda)  = \frac12 \beta(0) \kappa_0^2,
\end{align*}
contradicting \eqref{mon}. Note that the role of the value $\kappa_0$ in the latter computation is immaterial as one can argue with any $\kappa \not =0$.

\begin{proof}[Proof of Proposition \ref{non}]
Let us show that   $E_0$ cannot be minimized on $\P \times\A$. We firstly remark that
\beq
\label{inf}
E_0(\rho,\kappa) \stackrel{\eqref{mon}}{\geq}
E_0(0,\kappa)\stackrel{\text{Jensen}}{\geq} E_0(0,\kappa_0) \qquad \forall
(\rho,\kappa)\in \P\times \A.
\eeq
In fact, the first inequality is strict as soon as $\rho \kappa  \not \equiv 0$ almost everywhere while the second one is strict as soon as $\kappa$ is not constantly equal to $\kappa_0$ (recall that $\beta(0)>0$).
For all $\lambda\in(0,1)$, we now define
\begin{align*}
&\rho_\lambda (s) =
\left\{
  \begin{array}{ll}
     0\quad &\text{for} \  \ s \in [0,\lambda \pi]\\
\rho_0/(1-\lambda)&\text{for} \ \ s \in (\lambda \pi, \pi]\\
  0\quad &\text{for} \  \ s \in (\pi,(1+\lambda) \pi]\\
\rho_0/(1-\lambda)&\text{for} \ \ s \in ((1+\lambda) \pi, 2 \pi],\\
  \end{array}
\right.\qquad
\kappa_\lambda (s) =
\left\{
  \begin{array}{ll}
     \kappa_0/\lambda\quad &\text{for} \  \ s \in [0,\lambda \pi]\\
0&\text{for} \ \ s \in (\lambda \pi, \pi]\\
  \kappa_0/\lambda\quad &\text{for} \  \ s \in (\pi,(1+\lambda) \pi]\\
0&\text{for} \ \ s \in ( (1+\lambda) \pi, 2 \pi].\\
  \end{array}
\right.
\end{align*}
We now check that $(\rho_\lambda,\kappa_\lambda)\in \P \times \A$. Indeed,
\begin{align*}
&\int_0^{2\pi} \rho_\lambda \dint s = 2\pi
(1-\lambda)  \frac{\rho_0}{1-\lambda}  = 2\pi\rho_0 =
  2\pi,\qquad \int_0^{2\pi} \kappa_\lambda \dint s = 2
\lambda \pi \frac{\kappa_0}{ \lambda}  = 2\pi\kappa_0 = 2\pi.
\end{align*}
Moreover, by letting $\dsp{K_\lambda(s) = \int_0^s \kappa_\lambda(r)\, \dint r}$, namely,
\[
K_\lambda(s)=\left\{
 \begin{array}{ll}
     \kappa_0s/\lambda\quad &\text{for} \  \ s \in [0,\lambda \pi]\\
\kappa_0 \pi&\text{for} \ \ s \in (\lambda \pi, \pi]\\
  \kappa_0(s - (1-\lambda) \pi)/\lambda\quad &\text{for} \  \ s \in (\pi,(1+\lambda) \pi]\\
\kappa_02 \pi&\text{for} \ \ s \in ( (1+\lambda) \pi, 2 \pi],
  \end{array}
\right.
\]
we can compute
\begin{align*}
& \int_0^{2\pi}\cos \left(\int_0^s \kappa_\lambda(r)\, \dint r \right) \dint s
= \int_0^{2\pi} \cos(K_\lambda(s))\, \dint s
\\ & \quad
= \int_0^{\lambda \pi} \cos(\kappa_0s/\lambda)\, \dint s
+ \int_{\lambda \pi}^{\pi}\cos(\kappa_0 \pi)\, \dint s
+ \int_{\pi}^{(1+\lambda) \pi}\cos(\kappa_0(s-(1-\lambda) \pi)/\lambda)\, \dint s
+  \int_{(1+\lambda) \pi}^{2 \pi}\cos(2\kappa_0 \pi)\, \dint s
\\ & \quad
= \frac{\lambda}{\kappa_0}\sin\left(  \kappa_0  \pi \right)
-  \frac{\lambda}{\kappa_0}\sin 0+ \cos(\kappa_0 \pi)(1-\lambda)\pi
+ \frac{\lambda}{\kappa_0}\sin\left( 2\kappa_0 \pi\right)
- \frac{\lambda}{\kappa_0}\sin\left(\kappa_0 \pi \right) +\cos(2\kappa_0 \pi)(1-\lambda) \pi
\\   & \quad
=  \lambda\sin\pi - \lambda\sin 0 + (\cos\pi)(1-\lambda)\pi
+  \lambda\sin 2\pi - \lambda\sin\pi + (\cos 2\pi)(1-\lambda) \pi
=0,
\end{align*}
and analogously
\begin{align*}
& \int_0^{2\pi}\sin \left(\int_0^s \kappa_\lambda(r)\, \dint r \right) \dint s  = \int_0^{2 \pi} \sin(K_\lambda(s))\, \dint s
 \\
&\qquad
= -\lambda\cos\pi  + \lambda\cos 0 + (\sin \pi) (1-\lambda) \pi
-  \lambda\cos 2\pi + \lambda\cos\pi + (\sin 2\pi) (1-\lambda) \pi
=0.
\end{align*}
The latter ensures in particular that $(\rho_\lambda,\kappa_\lambda) \in \P \times \A$.

Let us now compute
\[
E_0(\rho_\lambda,\kappa_\lambda) = \lambda E_0(0,\kappa_0/\lambda) +
(1-\lambda)E_0(\rho/(1-\lambda),0) =  \lambda E_0(0,\kappa_0/\lambda)
\]
and note that $E_0(\rho_\lambda,\kappa_\lambda) \to E_0(0,\kappa_0) $ as $\lambda \to 1$.
Owing to \eqref{inf}, this entails that $E_0(\rho_\lambda,\kappa_\lambda)$
is an infimizing sequence on $\P \times \A$. On the other hand, the value $E_0(0,\kappa_0)$ cannot be reached in $\P \times \A$. Indeed, assume by contradiction to have
$(\rho,\kappa)\in \P\times \A$ with $E_0(\rho,\kappa)=E_0(0,\kappa_0)$. Recalling \eqref{inf}, we have that $\rho\kappa=0$ almost everywhere and $\kappa=\kappa_0$. This entails that $\rho=0$ almost everywhere so that necessarily $(\rho,\kappa)=(0,\kappa_0)$, which however does not belong to $\P \times \A$.
\end{proof}

Despite the lack of lower semicontinuity and the possible nonexistence of minimizers of variational problems, in some cases information may still be retrieved by analyzing the structure of infimizing sequences, see \cite{ball_fine_1987}. This perspective seems however to be of little relevance here. Assume  $(\rho,\kappa)$ to be a minimizer of $E_0$ in $\P \times \A$ and let $(\rho_\#,\kappa_\#)$ denote its periodic extension to $\mathbb R$. Let the \emph{fine-scaled} trajectories
\[
\rho_n(s) = \rho_\#(ns), \ \ \kappa_n(s) = \kappa_\#(ns)\ \ \forall s \in [0,2\pi]
\]
be defined. One may check that $(\rho_n,\kappa_n)\in \P \times \A $ as well and that
$E_0(\rho_n,\kappa_n)=E_0(\rho,\kappa)$, so that all $(\rho_n,\kappa_n)$ are minimizers (infimizing, in particular). On the other hand, $(\rho_n,\kappa_n) $ weakly converges to its mean
$(\rho_0,\kappa_0)$. This shows that, the limiting behavior of infimizing sequences may deliver scant information, for we recover the trivial state.

These facts motivate our interest for focusing on the case $\veps >0$ in the minimization problem
\eqref{eq:Egamin}. In contrast to the case $\veps=0$ of Proposition \ref{non}, energy $E_\veps$ can be minimized in $\P\times\A$ for all $\veps>0$.

\begin{prop}[Existence for $\veps>0$]
  Let $\veps>0$. Then, the minimization problem \eqref{eq:Egamin} admits a solution.
\end{prop}

\begin{proof}
  This is an immediate application of the Direct Method. Let
  $(\rho_n,\kappa_n)\in \P\times \A$ be an infimizing sequence for
  $E_{\veps}$ (such a sequence exists, for $E_\veps(\rho_0,\kappa_0)
  >-\infty$). We can assume with no loss of generality that $\sup
  E_\veps(\rho_n,\kappa_n)<\infty$. In particular, as $\beta\geq
  \beta_m>0$ we have that $\rho_n$ and $\kappa_n$ are uniformly
  bounded in $H^1(\torustwopi)$ and in $L^2(\torustwopi)$, respectively. This
  implies, at least for a not relabeled subsequence, that $\rho_n
  \rightharpoonup \rho $ in  $H^1(\torustwopi)$ hence strongly in 
  $C(\torustwopi)$ 
  and $\kappa_n \rightharpoonup \kappa$ in $L^2(\torustwopi)$. 
  We can hence pass to the limit in the relations
\[
\int_0^{2\pi}\rho_n\, \dint s = 2\pi, \quad
\int_0^{2\pi}\vect{\cos\left(\int_{0}^s\kappa_n(t)\dint t\right) \\
  \sin\left(\int_{0}^s\kappa_n(t)\dint t\right)}\dint s=\vect{0\\ 0},
\quad \int_0^{2\pi}\kappa_n\,\dint s=2\pi
\]
and obtain that $(\rho,\kappa)\in \P \times \A$ as well.
Moreover, $\beta(\rho_n)\to \beta(\rho)$ strongly in 
 $C(\torustwopi)$ 
 as $\beta$ is locally Lipschitz continuous. This implies that $(\beta(\rho_n))^{1/2}\kappa_n \rightharpoonup (\beta(\rho))^{1/2}\kappa $ in $L^2(\torustwopi)$ and lower semicontinuity ensures that
$E_\mu(\rho,\kappa) \leq \liminf_{n\to \infty} E_\mu(\rho_n,\kappa_n) = \inf E_\mu$,
so that $(\rho,\kappa)$  is a solution of problem \eqref{eq:Egamin}.
\end{proof}

The parameter $\veps$ is a datum of the problem and it is in particular
related to the characteristic length scale at which $\rho$ changes
along the curve. If $\veps$ is chosen to be large compared with the
length of the curve, the minimizer is again forced to be trivial.  Let
us make these heuristics precise in the following.

\begin{prop}[Trivial minimizer for $\veps$ large]
	\label{thm:trivial_unique_large_epsilon}
	For $\veps$ large enough, the trivial state $(\rho_0, \theta_0)$
        is the unique solution of the minimization problem
        \eqref{eq:Egamin}. 
\end{prop}

\begin{proof} We structure the proof into two steps. In Step 1 we show that, for
$\veps$ large, the trivial state $u_0 = (\rho_0, \theta_0)$ with $\rho_0=1$ and $\theta_0(s)=s$ is a strict minimizer in a neighborhood which is independent of $\veps$. In Step 2, we prove that all
minimizers converge to $u_0$ in the $H^1$ norm as $\veps \rightarrow\infty$. The combination of these two steps entails then that all minimizers necessarily coincide with $u_0$ for
$\veps$ sufficiently large, for they are arbitrarily close to $u_0$ (Step 2) which is locally the unique minimizer (Step 1).

\emph{Step 1: The trivial state is a  strict local minimizer.} Let us check that, for $\veps$ large enough, the second variation $ \delta^2 E_\veps(u_0)$ of $E_\veps=\frac12 \int_0^{2\pi}(\beta(\rho) \dot\theta^2 + \veps \dot\rho^2)\dint s$  is
positive. Indeed, for the arbitrary directions $u_1=(\rho_1,\theta_1)$ and $\tilde u_1=(\tilde\rho_1,\tilde\theta_1)$, we can compute
\[
\delta^2 E_\veps(u_0)(u_1,\tilde u_1)
=\int_0^{2\pi}\left(\frac{1}{2}\beta''(\rho_0)\,\dot\theta_0^2\,\rho_1\tilde\rho_1
+\beta'(\rho_0)\dot\theta_0\Big(\dot\theta_1\tilde\rho_1+\rho_1\dot{\tilde\theta}_1\Big)
+\beta(\rho_0)\,\dot\theta_1\dot{\tilde\theta}_1
+\veps^2\,\dot\rho_1\dot{\tilde\rho}_1
\right)\dint s,
\]
which is uniformly continuous around $u_0$. In particular, with $\dot\theta_0=1$ and rearranging terms,
\[
\delta^2 E_\veps(u_0)(u_1, u_1)
= \int_0^{2\pi}\left(
\veps \dot\rho_1^2 + \beta(\rho_0) \dot \theta_1^2
+ 2 \beta'(\rho_0) \rho_1 \dot \theta_1 + \frac{\beta''(\rho_0)}{2} \rho_1^2
\right)\dint s.
\]
By integrating by parts and using the Cauchy-Schwarz inequality in the third term we get
\begin{align*}
\delta^2 E_\veps(u_0)(u_1, u_1)
& \ge  \veps \int_0^{2\pi}\dot\rho_1^2 \,\dint s +  \beta(\rho_0) \int_0^{2\pi}\dot\theta_1^2 \,\dint s \\
&\quad{} - \left(\frac{4 \beta'(\rho_0)^2}{C \beta(\rho_0)}\right)^{\frac{1}{2}} \|\dot \rho_1\|_{L^2(0,2\pi)} \left(C \beta(\rho_0)\right)^{\frac{1}{2}} \|\theta_1\|_{L^2(0,2\pi)} - \frac{|\beta''(\rho_0)|}{2}\int_0^{2\pi} \rho_1^2\,\dint s,
\end{align*}
where $C$ is the Poincar\'{e} constant on $(0,2\pi)$. Using again Poincar\'{e}'s inequality to bound the second and last term in the right-hand side above, and Young's inequality for the third term we are left with
\[
\delta^2 E_\veps(u_0)(u_1, u_1)  \ge \left(\veps - \frac{2 \beta'(\rho_0)^2}{C \beta(\rho_0)} - \frac{|\beta''(\rho_0)|}{2C}\right) \int_0^{2\pi}\dot \rho_1^2\,\dint s+ \frac{\beta(\rho_0)}{2} \int_0^{2\pi}\dot \theta_1^2\,\dint s,
\]
which is positive for
\[
\veps > \frac{2\beta'(\rho_0)^2}{C\beta(\rho_0)} + \frac{|\beta''(\rho_0)|}{2C}\,.
\]
As $\delta^2E_\veps(u_0)$ is positive,  $u_0$ minimizes $E_\veps$ on some
neighborhood $U_\veps \subset \mathscr{P} \times \mathscr{A}$ for $\veps
\ge \veps_0$ and for some $\veps_0 > 0$. Since $E_\veps$ is increasing in
$\veps$ and $E_\veps(u_0)$ does not depend on $\veps$, $U_\veps$ may be taken
to be increasing in $\veps$ as well. Thus, $u_0$ minimizes $E_\veps$ on $U_{\veps_0}$ for all $\veps \ge \veps_0$.

\emph{Step 2: Global minimizers converge to the trivial state.}
We next prove that, for  any $\delta > 0$ there exists $\veps_c > 0$ such that for any $\veps > \veps_c$, any global minimizer $(\rho,\theta)$ of $E_\veps$ is such that
\begin{equation}
\label{eq:estimate}
\| \rho - \rho_0\|_{L^2(0,2\pi)} + \|\theta - \theta_0\|_{L^2(0,2\pi)}\lesssim \|\dot\rho\|_{L^2(0,2\pi)} + \|\dot\theta - \dot\theta_0\|_{L^2(0,2\pi)} < \delta
\end{equation}
where we use the sign $\lesssim$ to indicate the implicit occurrence of a constant just depending on data. In fact, we have that
\begin{align}
E_\veps(\rho_0,\theta_0) \ge E_\veps(\rho,\theta)
& = \int_0^{2\pi} \left(\frac12 \beta(\rho) \dot \theta^2 + \frac{\veps}{2} \dot \rho^2 \right) \dint s
 \nonumber\\
& \ge
\int_0^{2\pi} \left( \frac12 \beta_m \dot \theta^2 + \frac{\veps}{2} \dot \rho^2 \right)\dint s
\ge \int_0^{2\pi}\left(\frac12 \beta_m  \dot \theta^2_0+
  \frac{\veps}{2} \dot \rho^2 \right) \dint s ,\label{torefine}
\end{align}
since $\theta_0$ minimizes the Dirichlet energy $\int_0^{2\pi}\dot \theta^2 \, \dint s $ under the conditions $\theta(0) = 0$, $\theta(2\pi) = 2\pi$. Since $E_\veps(\rho_0,\theta_0)=\frac12\int_0^{2\pi}\beta(\rho_0)\dot\theta_0^2=\pi\beta(\rho_0)<\infty$, both terms in the above right-hand side are bounded. We hence deduce that
$\int_0^{2\pi} \dot \theta^2\, \dint s$ is bounded uniformly in $\veps$ and
 $\int_0^{2\pi} \dot \rho^2\, \dint s = O(\veps^{-1}) = o(\veps^{-1/2})$, so that there exists $\veps_1 > 0$ such that for $\veps > \veps_1$ we have $\int_0^{2\pi} \dot \rho^2 \, \dint s < \delta/2$.
Now,
$\rho\in\mathscr{P}$ implies $\int_0^{2\pi}(\rho-\rho_0)\dint s=0$ 
and by the Poincar\'e inequality as well as the continuous embedding
in $L^\infty(0,2\pi)$,
\[
\| \rho - \rho_0\|_{L^\infty(0,2\pi)}\lesssim \| \dot \rho\|_{L^2(0,2\pi)}= o(\veps^{-1/4}),
\]
and, by the local Lipschitz continuity of $\beta$,
\[
\| \beta(\rho) - \beta(\rho_0)\|_{L^\infty(0,2\pi)}= o(\veps^{-1/4}).
\]
This allows us to refine estimate \eqref{torefine} as follows:
\[
E_\veps(\rho_0,\theta_0) \ge E_\veps(\rho,\theta)
\ge \int_0^{2\pi} \left(\frac12 \beta(\rho_0) \dot \theta^2_0 + \frac{\veps}{2} \dot \rho^2 \right)\dint s
+ o(\veps^{-1/4}) = E_\veps(\rho_0,\theta_0)+ \frac{\veps}{2} \int_0^{2\pi} \dot \rho^2 \dint s + o(\veps^{-1/4}),
\]
from which we get $\lim\limits_{\veps \rightarrow \infty}\veps\int_0^{2\pi} \dot \rho^2 \, \dint s= 0$, and then
\[
\lim\limits_{\veps \rightarrow \infty} E_\veps(\rho,\theta)= \lim\limits_{\veps \rightarrow \infty}  \int_0^{2\pi}\frac12\beta(\rho) \dot \theta^2 \, \dint s= 	E_\veps(\rho_0,\theta_0).
\]
Finally, we control
\[
\left| \int_0^{2\pi} \frac12\beta(\rho) \dot\theta^2 \, \dint s- E_\veps(\rho_0,\theta_0)\right|
= \left|\int_0^{2\pi} \frac{\beta(\rho_0)}{2}  (\dot\theta^2 - \dot\theta_0^2) \dint s+ o(\veps^{-1/4})\right|,
\]
so to prove that
$\lim\limits_{\veps\to \infty} \int_0^{2\pi} \dot \theta^2 \, \dint s = \lim\limits_{\veps\to \infty}  \int_0^{2\pi} \dot \theta_0^2 \, \dint s = 2\pi$. This is enough to conclude that
\[
\lim\limits_{\veps \rightarrow \infty} \|\dot\theta - \dot\theta_0\|_{L^2(0,2\pi)} = 0.
\]
We can  then choose $\veps_2$ such that, for $\veps > \veps_2$, $\|\dot\theta - \dot\theta_0\|_{L^2(0,2\pi)} < \delta/2$ and set $\veps_c = \max\{\veps_1, \veps_2\}$ for which the second inequality in \eqref{eq:estimate} holds. The first inequality follows from Poincar\'e's inequality.
\end{proof}

We present now a symmetry result which will turn out useful later on, when interpreting the numerical findings.

\begin{prop}[Symmetry of $E_\veps$]
	\label{prop:m_symmetry}
	If $(\rho, \theta)$ is a local minimizer of $E_\veps$ for $\beta$, then $(2\rho_0 - \rho, \theta)$ is a local minimizer of $E_\veps$ for $\tilde\beta$, defined as $\tilde\beta(\rho) = \beta(2\rho_0 - \rho)$.
\end{prop}

\begin{proof}
The integrand is unchanged by this transformation, so that the first and second variations of $E_\veps$ at $(\rho,\theta)$ and $(2\rho_0 - \rho, \theta)$ when considering respectively $\beta$ and $\tilde\beta$ are identical.
\end{proof}

\section{Bifurcation analysis}\label{sec:bifurcation}

By Proposition \ref{thm:trivial_unique_large_epsilon}, the circle with constant density is the global minimizer of the energy $E_\veps$ \eqref{eq:Ereg}
for large enough values of the diffusivity $\veps$. In this section candidates for nontrivial minimizers are constructed by bifurcation from this trivial
critical point with decreasing $\veps > 0$ as bifurcation parameter. The analysis will be based on the Euler-Lagrange equations of a suitable Lagrangian,
incorporating the constraints of closedness of the curve and of given total mass. Additional auxiliary conditions will eliminate symmetries resulting from
arbitrary positioning of the curve in the plane.


\subsection{Euler-Lagrange equations}

We introduce the Lagrange multipliers $(\lambda_x,\lambda_y)\in \mathbb{R}^2$ for the closedness constraint and $\lambda_M \in \mathbb{R}$ for the mass constraint (cf.\ the definitions \eqref{eq:Athetareg} and \eqref{eq:Adensityreg} of admissible $\theta$ and $\rho$ respectively), and define the
Lagrangian
\[
  \mathcal{L}(\bar u) := \frac{1}{2} \intfull{\left(\beta(\rho)\dot\theta^2 + \veps \dot\rho^2\right)}
  + \intfull{(\lambda_x \cos \theta + \lambda_y \sin \theta + \lambda_M(\rho - 1))} \,,
\]
with $\bar u=(\rho,\theta,\lambda_x,\lambda_y,\lambda_M)$, where $\rho, \theta - s \in H^1(\mathbb{T}_{2\pi})$.
Critical points of the energy subject to the constraints solve the \emph{Euler-Lagrange equations}
\begin{align}
\veps\ddot\rho-\frac{1}{2}\,\beta'(\rho)\,\dot\theta^2&=\lambda_M \,,\label{eq:ELrho}\\
\dds\left(\beta(\rho)\,\dot\theta\right)&=-\lambda_x\sin\theta+\lambda_y\cos\theta \,, \label{eq:ELtheta}
\end{align}
along with the \emph{boundary conditions} 
\beq\label{eq:BCs}
\rho(2\pi) - \rho(0) = 0\,, \quad \dot\rho(2\pi) - \dot\rho(0) = 0\,, \quad \theta(0)=0\,, \quad \theta(2\pi) = 2\pi\,,
\eeq
and mass and closedness  \emph{constraints}
\beq\label{eq:constraints}
\intfull{\rho}=2\pi \,,\qquad \int_0^{2\pi}\vect{\cos\theta \\ \sin\theta}\dint s=\vect{0\\ 0} \,,
\eeq
respectively.
For every solution, new solutions can be produced by 
arbitrary shifts in $s$. In order to eliminate this degree of freedom we add the condition
\beq\label{eq:theta0}
   \rho(0)=1 \,,
\eeq
where we note that 1 is the average value of $\rho$ by the mass constraint, which is assumed by every continuous solution. One symmetry remains:
The problem is still invariant under the flip symmetry $s\leftrightarrow -s$ (with $\theta\leftrightarrow -\theta$, $\lambda_y\leftrightarrow -\lambda_y$).

The \emph{trivial solution} $\bar u_0$ of \eqref{eq:ELrho}--\eqref{eq:theta0} (and the minimizer for large enough $\veps$) is the unit circle with constant density:
\beq\label{triv-sol}
   \theta_0(s) = s \,,\quad \rho_0(s) = 1 \,,\quad \lambda_{x0} = \lambda_{y0} = 0 \,,\quad \lambda_{M0} =  -\frac{1}{2}\beta'(1) \,.
\eeq
For the stiffness coefficient $\beta$ we shall assume the following local behavior close to the trivial solution:
\beq\label{eq:beta-ass}
  \beta(\rho) = 1 + m(\rho-1) + h \frac{(\rho-1)^2}{2} + O\left((\rho-1)^5\right) \qquad\mbox{as } \rho\to 1 \,.
\eeq
The bifurcation behaviour will be characterized in terms of the Taylor coefficients $m,h\in\mathbb{R}$. The complexity of the  bifurcation computations
below have motivated the simplifying assumption that the third- and fourth-order coefficients vanish.
Including these higher-order terms would only alter the sub-/supercritical nature of the bifurcation, but not the critical values for the parameter $\mu$.

\subsection{Linearization around the trivial state}\label{sec:lin}

In terms of a small correction $\bar u_1 := \bar u - \bar u_0$, the linearization of problem \eqref{eq:ELrho}--\eqref{eq:theta0} reads (using \eqref{eq:beta-ass})
\begin{align}
  \veps\ddot \rho_1(s)-m\,\dot\theta_1(s)-\frac{h}{2}\,\rho_1(s) - \lambda_{M1} &= f(s)\,,\label{eq:ELone}\\
  \dds\left(\dot\theta_1(s)+m\,\rho_1(s)\right) + \lambda_{x1}\sin s - \lambda_{y1}\cos s &= g(s)\,,\label{eq:ELtwo}
\end{align}
subject to the boundary conditions
\beq\label{eq:BCs_lin}
\rho_{1}(2\pi) - \rho_{1}(0) = 0\,, \quad \dot\rho_{1}(2\pi) - \dot\rho_{1}(0) = 0\,, \quad \theta_{1}(0)=0\,, \quad \theta_{1}(2\pi) = 0\,,
\eeq
the constraints
\beq\label{eq:constraints_lin}
\int_0^{2\pi}\rho_1(s)\,\dint s=0 \,,\qquad   \int_0^{2\pi}\vect{-\sin s \\ \:\:\:\cos s}\theta_1(s)\dint s=\vect{\alpha_x\\ \alpha_y} \,,
\eeq
and the auxiliary condition
\beq\label{eq:theta0_lin}
\rho_1(0) = 0 \,.
\eeq
The inhomogeneities $f(s), g(s), \alpha_x, \alpha_y$ can be interpreted as nonlinear corrections.

\begin{prop}[Solution of the homogeneous linearized system]\label{prop:lin}
A nonzero solution of \eqref{eq:ELone}--\eqref{eq:theta0_lin} with $f=g=\alpha_x=\alpha_y=0$, $\mu>0$, $m,h\in\mathbb{R}$, only exists in the following
cases:
\begin{description}
    \item[Case 1:] There exists $j\in\mathbb{N}$, $j\ge 2$, such that $\veps=\veps_j(m,h) := \frac{1}{j^2}\left(m^2-\frac{h}{2}\right) \ne -\frac{h}{2}$. The space
of solutions is one-dimensional and given by
\beq\label{u11}
   \rho_1(s) = a_1 \sin(js) \,,\quad \theta_1(s) = \frac{a_1 m}{j}(\cos(js)-1) \,,\quad \lambda_{x1}=\lambda_{y1}=\lambda_{M1}=0 \,,\qquad a_1\in\mathbb{R}\,.
\eeq
\item[Case 2:] $\veps=\veps_1(h) := -\frac{h}{2}\neq\frac{1}{j^2}\left(m^2-\frac{h}{2}\right)$ for all $j\in\mathbb{N}$, $j\ge 2$. The space of solutions
is one-dimensional and given by
\beq\label{u12}
   \rho_1(s) = b_1 \sin s \,,\quad \theta_1(s) = 0 \,,\quad \lambda_{x1}=\lambda_{M1}=0 \,,\quad \lambda_{y1} = b_1 m \,,\qquad b_1\in\mathbb{R}\,.
\eeq
\item[Case 3:] There exists $j\in\mathbb{N}$, $j\ge 2$, such that $\veps=\veps_j(m,h) = \veps_1(h)$. The space of solutions
is two-dimensional and given by
\[
   \rho_1(s) = a_1 \sin(js) + b_1 \sin s \,,\quad \theta_1(s) = \frac{a_1 m}{j}(\cos(js)-1) \,,\quad \lambda_{x1}=\lambda_{M1}=0 \,,\quad
   \lambda_{y1} = b_1 m \,,\qquad a_1,b_1\in\mathbb{R}\,.
\]
\end{description}
\end{prop}

\begin{remark}\label{rem:lin}
\begin{enumerate}
\item As expected, bifurcations only occur under the condition (see \eqref{beta2})
\[
   2m^2-h = 2\beta'(1)^2 - \beta(1)\beta''(1)>0 \,.
\]
\item For Case 1 solutions, the curvature correction $\kappa_1 = \dot\theta_1$ satisfies $\kappa_1 = -m\rho_1$. This means that the sign of $m=\beta'(1)$
decides if curvature maxima coincide with density maxima ($m<0$) or with density minima ($m>0$), which is not a surprising result. The condition $j\ge 2$
is a manifestation of the Four Vertex Theorem (see Section \ref{sec:background}).
\item Case 2 solutions exhibit only one maximum and one minimum of the density without any effect on the circular shape of the curve. The numerical
computations reported in Section \ref{sec:numerics} show, however, that these solutions initiate bifurcating branches strongly deviating from the circular
shape far enough from the bifurcation point.
\item Whereas Cases 1 and 2 correspond to codimension-one bifurcations, Case 3 represents a bifurcation of codimension two, whose nonlinear
structure will not be analyzed in the following.
\item In a bifurcation scenario, where the values of $m$ and $h$ are fixed and the value of $\mu$ is decreased, there are several situations. For $h\ge2m^2$
no bifurcations occur by convexity (see above). For $0\le h < 2m^2$ an infinite series of Case 1 bifurcations occurs at the bifurcation
values $\mu_j$, $j\ge 2$, with the first one at $\mu=\mu_2$. For $h<0$, apart from the bifurcations at $\mu=\mu_2,\mu_3,\ldots$, there is
also a Case 2 bifurcation at $\mu=\mu_1$. There are two subcases concerning the question, which bifurcation occurs first, determined by the criterion
\[
    \mu_1 > \mu_2 \quad\Longleftrightarrow\quad h < -\frac{2}{3}m^2 \,.
\]
Codimension-two bifurcations occur whenever $\mu_1=\mu_j$, i.e. $h = -2m^2/(j^2-1)$ for some $j\ge 2$. These observations are illustrated in the
$(m,h)$-plane in Figure \ref{fig:cases}.
\end{enumerate}
\end{remark}

\begin{figure}
\centering
\includegraphics{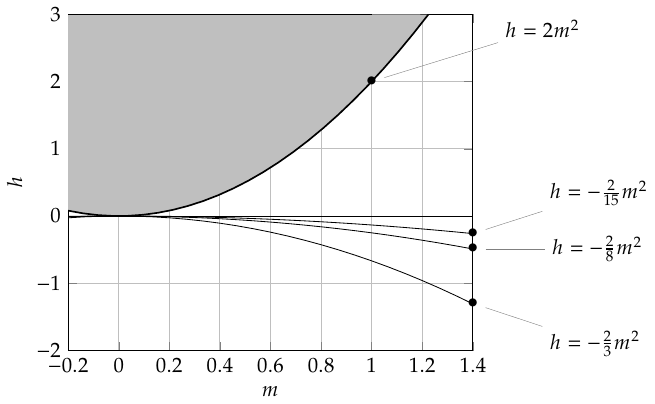}






\caption{Regions of different bifurcation behaviour in the $(m,h)$-plane according to Proposition \ref{prop:lin}: No bifurcations above the parabola $h=2m^2$.
Only Case 1 bifurcations between the parabola and the $m$-axis. Case 1 and Case 2 bifurcations below the $m$ axis, with codimension-two bifurcations on
the parabolas $h=-2m^2/(j^2-1)$, $j\ge 2$. The first bifurcation is a Case 2 bifurcation below the parabola $h=-2m^2/3$, and a Case 1 bifurcation with $j=2$
otherwise.}
\label{fig:cases}
\end{figure}

\begin{proof}
By the smoothness of solutions of ordinary differential equations we can employ Fourier representation and write
\[
   \rho_1(s) = \sum_{k\in\mathbb{Z}} \widehat\rho_{1,k} e^{iks} \,,\qquad \theta_1(s) = \sum_{k\in\mathbb{Z}} \widehat\theta_{1,k} e^{iks} \,.
\]
We keep the inhomogeneities for the moment, since this will be useful for the proof of the following result, and we use their Fourier series
\[
   f(s) = \sum_{k\in\mathbb{Z}} \widehat f_k e^{iks} \,,\qquad g(s) = \sum_{k\in\mathbb{Z}} \widehat g_k e^{iks} \,.
\]
The constraints \eqref{eq:constraints_lin} imply
\beq\label{Fourier-constr}
   \widehat\rho_{1,0} = 0 \,,\qquad \widehat\theta_{1,\pm 1} = \frac{1}{2\pi}(\alpha_y \pm i\alpha_x) \,.
\eeq
Comparing Fourier coefficients for $k=0$ in \eqref{eq:ELone}, \eqref{eq:ELtwo} implies
\beq\label{Fourier0}
   \lambda_{M1} = -\widehat f_0 \qquad\mbox{and}\qquad \widehat g_0=0 \,,
\eeq
where the latter has to be seen as a solvability condition for the inhomogeneous problem. Coefficients for $k=\pm 1$ in \eqref{eq:ELone} and
\eqref{eq:ELtwo} give
\begin{align}\label{Fourier1rho}
   &-\left(\veps - \veps_1\right)\widehat\rho_{1,\pm 1} = \widehat f_{\pm 1} - \frac{m}{2\pi}(\alpha_x \mp i\alpha_y) \,,\\
   &  \pm mi \widehat\rho_{1,\pm 1} \mp \frac{i}{2}\lambda_{x1} - \frac{1}{2}\lambda_{y1} = \widehat g_{\pm 1} + \frac{1}{2\pi}(\alpha_y \pm i\alpha_x)\,.\label{Fourier1theta}
\end{align}
For coefficients with $|k|\ge 2$ we obtain
\[
   - \left( \veps k^2 + \frac{h}{2}\right)\widehat\rho_{1,k} - ikm \widehat\theta_{1,k} = \widehat f_k \,,\qquad -k^2\widehat\theta_{1,k} + ikm\widehat\rho_{1,k} = \widehat g_k \,,
\]
implying
\beq\label{Fourierj}
   -k^2(\mu-\mu_{|k|})\widehat \rho_{1,k} = \widehat f_k - \frac{im}{k}\widehat g_k \,,\qquad \widehat\theta_{1,k} - \frac{im}{k}\widehat\rho_{1,k} = - \frac{1}{k^2}\widehat g_k\,.
\eeq
The results follow immediately from the homogeneous ($f=g=\alpha_x=\alpha_y=0$) versions of \eqref{Fourier-constr}--\eqref{Fourierj},
using the auxiliary conditions \eqref{eq:theta0_lin}.
\end{proof}

\begin{lemma}[Solvability conditions] \label{lem:lin}
\begin{description}
\item[Case 1:] Problem \eqref{eq:ELone}--\eqref{eq:theta0_lin} with $0<\mu=\mu_j\ne\mu_1$, $j\ge 2$, has a solution if and only if
\[
  j\intfull{f(s)\cos(js)} = m\intfull{g(s)\sin(js)} \,,
\] \beq\label{solv1}
   j\intfull{f(s)\sin(js)} = -m\intfull{g(s)\cos(js)} \,,\qquad\mbox{and}\qquad \intfull{g(s)} = 0\,.
\eeq
\item[Case 2:] Problem \eqref{eq:ELone}--\eqref{eq:theta0_lin} with $0<\mu=\mu_1\ne\mu_j$, $\forall\,j\ge 2$, has a solution if and only if
\beq\label{solv2}
  \intfull{f(s)\cos s} = m\alpha_x \,,\qquad \intfull{f(s)\sin s} = m\alpha_y \,,\qquad\mbox{and}\qquad\intfull{g(s)} = 0\,.
\eeq
\end{description}
\end{lemma}

\begin{proof}
For Case 1 the solvability conditions follow from \eqref{Fourier0}, \eqref{Fourierj}, and for Case 2 from \eqref{Fourier0}, \eqref{Fourier1rho}.
\end{proof}

\subsection{Asymptotic expansion around bifurcation points}\label{sec:hot}

For the codimension-one bifurcations identified above (Cases 1 and 2 in Proposition \ref{prop:lin}), the existence of bifurcating solution branches is
guaranteed by general results on bifurcations from simple eigenvalues \cite{crandall_bifurcation_1971}. The local shape of these branches will be analyzed by
perturbation expansions. By the presence of a flip symmetry in problem \eqref{eq:ELrho}--\eqref{eq:theta0}, pitchfork bifurcations can be expected, at least
generically. For the bifurcation at $\mu=\mu_j$, $j\in\mathbb{N}$, we therefore introduce
\[
   \mu = \mu_j - \sigma A^2 \,,\qquad 0<A\ll 1\,,\quad \sigma\in \{1,-1\} \,.
\]
The small parameter $A$ measures the distance from the bifurcation point, whereas the sign $\sigma$, to be determined by the analysis, tells us
whether the bifurcation is \emph{supercritical} for $\sigma>0$ or \emph{subcritical} for $\sigma<0$. This convention is in line with the scenario of
decreasing $\mu$ (see Remark \ref{rem:lin}, 5.). The solution $\bar u=(\rho,\theta,\lambda_x,\lambda_y,\lambda_M)$ of \eqref{eq:ELrho}--\eqref{eq:theta0}
will be approximated by an asymptotic expansion
\beq
\label{eq:perturbationansatz}
  \bar u = \bar u_0 + A \bar u_1 + A^2 \bar u_2 + A^3 \bar u_3 + O(A^4) \,,
\eeq
where the reason for going up to third order will become apparent below.

\begin{remark}[Bifurcation diagram for classical elasticae]
The expectation of pitchfork bifurcations and, thus, the ansatz \eqref{eq:perturbationansatz} can also be motivated by the bifurcation diagram for classical elasticae (e.g.\ \cite[Ch.\ 7]{marsden_mathematical_1994}). The diagram shows an infinite series of bifurcations similar to the series of Case 1 bifurcations
in \eqref{eq:ELrho}--\eqref{eq:theta0}. In the classical elastica problem all these bifurcations are supercritical pitchforks.
\end{remark}

The notation in \eqref{eq:perturbationansatz} is consistent with the above. The trivial rotationally symmetric solution of \eqref{eq:ELrho}--\eqref{eq:theta0}
is denoted by $\bar u_0$, and the first correction $\bar u_1$ has to satisfy the homogeneous version ($f=g=\alpha_x=\alpha_y=0$) of the linearized problem \eqref{eq:ELone}--\eqref{eq:theta0_lin} with $\mu=\mu_j$, whose
solution is unique up to a scalar constant ($a_1$ in Case 1 and $b_1$ in Case 2). The problems for $\bar u_2$ and $\bar u_3$ are determined by substituting the
ansatz \eqref{eq:perturbationansatz} into \eqref{eq:ELrho}--\eqref{eq:theta0}, expanding the nonlinearities and comparing coefficients of $A^2$ and $A^3$.
Both $\bar u_2$ and $\bar u_3$ solve inhomogeneous versions of the linearized problem \eqref{eq:ELone}--\eqref{eq:theta0_lin} with $\mu=\mu_j$ and
with the inhomogeneities
\begin{align}
& f_2 = \frac{m}{2}\dot\theta_1^2 + h\rho_1\dot\theta_1\,,\qquad
  g_2 = -  \dds\Big(m\rho_1\dot\theta_1+{\frac{h}{2}}\rho_1^2\Big) -\theta_1(\lambda_{x1}\cos s+\lambda_{y1}\sin s) \,,\label{eq:inhom21}\\
& \vect{\alpha_{x2}\\ \alpha_{y2}} = \frac{1}{2}\intfull{\theta_1^2\vect{\cos s\\ \sin s}} \,,\label{eq:inhom22}
\end{align}
for $\bar u_2$, and
\begin{align}
& f_3 = \sigma \ddot\rho_1 + m\dot\theta_1\dot\theta_2 +
  \frac{h}{2}\Big(2\rho_2\dot\theta_1+2\rho_1\dot\theta_2+\rho_1\dot\theta_1^2\Big) \,,\label{eq:inhom31}\\
& g_3
   = - \dds\Big(m\rho_1\dot\theta_2+m\rho_2\dot\theta_1+\txt{\frac{h}{2}}\rho_1^2\dot\theta_1+h\rho_1\rho_2  \Big) \nonumber\\
   &\quad -\theta_1(\lambda_{x2}\cos s+\lambda_{y2}\sin s)-\theta_2(\lambda_{x1}\cos s+\lambda_{y1}\sin s)
       +\frac{1}{2}\theta_1^2(\lambda_{x1}\sin s-\lambda_{y1}\cos s) \,,\label{eq:inhom32}\\
& \vect{\alpha_{x3}\\ \alpha_{y3}}
   = \intfull{\vect{\theta_1\theta_2\cos s - \frac{1}{6}\theta_1^3\sin s\\ \theta_1\theta_2\sin s + \frac{1}{6}\theta_1^3\cos s}} \,,\label{eq:inhom33}
\end{align}
for $\bar u_3$. Note that the
inhomogeneities depend on lower-order terms. So the terms in the asymptotic expansion \eqref{eq:perturbationansatz} can be computed recursively. However, this comes with two problems, which are connected: the solution of the linearized problem is not unique (see Proposition \ref{prop:lin}),
and it does not have a solution for arbitrary inhomogeneities (see Lemma \ref{lem:lin}). The strategy is to recover the lacking information for uniqueness from
the solvability conditions for higher-order problems.
It will turn out (as a consequence of the above mentioned flip symmetry) that the inhomogeneities \eqref{eq:inhom21}, \eqref{eq:inhom22} of the
second-order problem satisfy the solvability
conditions, no matter what the value of the missing first-order constant ($a_1$ in Case 1 and $b_1$ in Case 2) is. This is the reason why the
third-order problem has to be considered, whose solvability condition will provide an equation for the missing first-order constant. In the following the
essential results of these straightforward but lengthy computations will be given. They have been carried out manually and checked with the
help of {\tt MATHEMATICA}.

\subsection*{Case 1 bifurcations}

The goal is to determine the value of the constant $a_1$ in the first-order correction $\bar u_1$ of the expansion \eqref{eq:perturbationansatz}, given in
\eqref{u11}. The first step is the computation of the second-order terms.

\begin{lemma}[Case 1: second-order solution] \label{lem:Case0order2} Let $j\ge 2$ and $\bar u_1$ be given by \eqref{u11}. Then every solution of
\eqref{eq:ELone}--\eqref{eq:theta0_lin} with $0<\mu=\mu_j\ne\mu_1$ and with the inhomogeneities given by \eqref{eq:inhom21}, \eqref{eq:inhom22}
can be written as
\begin{align}
& \rho_2(s)=a_2\sin(js) + \frac{a_1^2(m(m^2-h))}{2(2m^2-h)}(\cos(2js)-\cos(js)) \,,\nn\\
& \theta_2(s)={a_2 \frac{m}{j}(\cos(js)-1)-\frac{a_1^2(6m^4-6m^2h+h^2)}{8j(2m^2-h)}\sin(2js)} \,,\qquad a_2\in\mathbb{R}\,,\label{eq:u21}\\
& \lambda_{x2}=\lambda_{y2}=0,\quad{\lambda_{M2}=-\frac{a_1^2 m(m^2-2h)}{4}} \,.\nn
\end{align}
\end{lemma}

\begin{lemma}[Case 1: the missing constant] \label{lem:Case0order3} Let $j\ge 2$, $0<\mu_j\ne\mu_1$, and $\bar u_1, \bar u_2$ be given by \eqref{u11},
\eqref{eq:u21}. Then the inhomogeneities given by \eqref{eq:inhom31}--\eqref{eq:inhom33} satisfy the solvability conditions \eqref{solv1}, if and only if
\beq
\label{eq:noresonance0}
a_1\left(j^2\sigma - a_1^2\frac{Z(m,h)}{8(2m^2-h)}\right)=0 \,,
\qquad\text{with}\quad
Z(m,h):=-14m^6+36m^4h-18m^2h^2+h^3 \,.\quad
\eeq
\end{lemma}

This shows that Case 1 bifurcations are pitchforks if and only if $Z(m,h)\ne 0$. The amplitude of the first-order term along the bifurcating branch is determined by
the nontrivial solutions of  \eqref{eq:noresonance0}:
\beq\label{eq:a1Case0}
a_1^2=\frac{8j^2\sigma(2m^2-h)}{Z(m,h)} \,,
\eeq
which shows for the criticality $\sigma = \mbox{sign }Z(m,h)$ that the bifurcation is \emph{supercritical} for $Z(m,h)>0$ and \emph{subcritical} for $Z(m,h)<0$.
Writing $Z(m,h)=m^6(z^3-18z^2+36z-14)$ with $z:=h/m^2<2$ shows that $Z(m,h)>0$ and $2m^2>h$ are equivalent to
\beq
\label{eq:z12}
z_1<z<z_2\quad\text{with}\quad z_1\approx 0.52\quad\text{and}\quad z_2\approx 1.71.
\eeq
Consequently, a supercritical pitchfork bifurcation occurs in the parabolic region
$\{(m,h)\in\RR^2:\:z_1m^2<h<z_2m^2\}$.
Conversely, if $(m,h)$ is such that $Z(m,h)<0$, which holds for $h<z_1m^2$ or $2m^2>h>z_2m^2$, then the bifurcation is subcritical. The situation is illustrated in Figure \ref{fig:Zmh}. Note that the criticality is independent from $j$. For a fixed pair $(m,h)$ the whole series of Case 1 bifurcations has
the same criticality.

\subsection*{Case 2 bifurcations}

\begin{figure}
\centering
\includegraphics{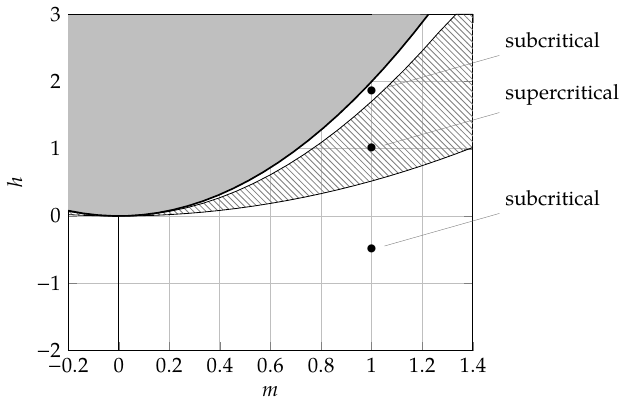}




\caption{Contour plot of $Z(m,h)$ given by  \eqref{eq:noresonance0}.  The solution in Case 1 has the structure of a supercritical pitchfork bifurcation whenever $Z(m,h)>0$ and $h<2m^2$. These conditions define the crosshatched region $z_1m^2<h<z_2m^2$ below the parabola $h=2m^2$ (black line); $z_1\approx 0.52$ and $z_2\approx 1.71$, see \eqref{eq:z12}. Conversely, if $Z(m,h)<0$ which is true when $h<z_1m^2$ or in the narrow white region given by $z_2m^2<h<2m^2$, then the bifurcation is subcritical.}
\label{fig:Zmh}
\end{figure}

\begin{lemma}[Case 2: second-order solution] \label{lem:Case2order2} Let $\bar u_1$ be given by \eqref{u12}. Then every solution of
\eqref{eq:ELone}--\eqref{eq:theta0_lin} with $0<\mu=\mu_1\ne\mu_j$, $\forall j\ge 2$, and with the inhomogeneities given by \eqref{eq:inhom21}, \eqref{eq:inhom22}
can be written as
\begin{align}
& \rho_2(s)=b_2\sin s + \frac {b_1^2 mh}{2(2m^2+3h)} (\cos(2s)-\cos s) \,,\qquad b_2\in\mathbb{R}\,,\nn\\
& \theta_2(s)=\frac{b_1^2 3h^2}{8(2m^2+3h)} \sin(2s) \,,\label{eq:u22}\\
& \lambda_{x2}= -\frac{b_1^2 m^2 h}{2(2m^2+3h)}\,,\quad \lambda_{y2}=b_2 m \,,\quad \lambda_{M2}=0 \,.\nn
\end{align}
\end{lemma}

\begin{lemma}[Case 2: the missing constant] \label{lem:Case2order3} Let $0<\mu_1\ne\mu_j$, $\forall j\ge 2$, and $\bar u_1, \bar u_2$ be given by \eqref{u12},
\eqref{eq:u22}. Then the inhomogeneities given by \eqref{eq:inhom31}--\eqref{eq:inhom33} satisfy the solvability conditions \eqref{solv2}, if and only if
\beq
\label{eq:noresonance11}
b_1\Big(\sigma +b_1^2\,\frac{3h^3}{8(2m^2+3h)}\Big)=0 \,.\quad
\eeq
\end{lemma}

This shows that Case 2 bifurcations are pitchforks, since
\[
    b_1^2 = -\frac{8\sigma(2m^2+3h)}{3h^3} \,
\]
is finite by $\mu_1 = -h/2 >0$ and nonvanishing by $2m^2+3h = 8(\mu_2-\mu_1) \ne 0$. The bifurcation is subcritical for $2m^2+3h<0$,
i.e. when the Case 2 bifurcation is the first one for decreasing $\mu$ (see Remark \ref{rem:lin}, 5.). It is supercritical for $2m^2+3h>0$.
It is also noteworthy that the circular shape of the trivial solution curve is now perturbed at the order of $A^2$ with the perturbation given in \eqref{eq:u22}.
The leading order density perturbation $\rho_1$ has both its extrema coinciding either with the maxima of the curvature perturbation $\dot\theta_2$
(in the subcritical case) or with its minima (in the supercritical case). This can be verified numerically, see Cases (iv) and (v) in Section~\ref{sec:numerics results} and Figure~\ref{fig:numerics shapes} for $j=1$.

\subsection{Energy and stability}

As an indicator for the stability of bifurcating solutions, we investigate the changes of the energy \eqref{eq:Ereg} along bifurcating branches.
For this purpose, we substitute $\mu = \mu_j - \sigma A^2$ and the asymptotic expansion \eqref{eq:perturbationansatz} of the bifurcating solution into the energy and re-expand:
\[
E_\veps(\rho,\theta)=\frac{1}{2}\int_0^{2\pi}(\beta(\rho)\,\dot\theta^2+\veps\dot\rho^2)\,\dint s
=E_0+AE_1+A^2E_2+A^3E_3+A^4E_4+ O(A^5) \,.
\]
For the coefficients we obtain (all computations of this section again verified with {\tt MATHEMATICA})
\begin{align}
E_0&=\frac{1}{2}\int_0^{2\pi}\dint s=\pi=E_{\mu_j}(\rho_0,\theta_0) \,,\qquad E_1=\frac{1}{2}\int_0^{2\pi}\left(m\rho_1+2\dot\theta_1\right)\dint s = 0 \,,
\label{eq:E01}\\
E_2&=\frac{1}{2}\int_0^{2\pi}\left(2m\rho_1\dot\theta_1+\txt{\frac{h}{2}}\rho_1^2+\dot\theta_1^2+\mu_j\dot\rho_1^2\right)\dint s \,,
\label{eq:E2}\\
E_3&=\frac{1}{2}\int_0^{2\pi}\left(2m\rho_1\dot\theta_2+2m\rho_2\dot\theta_1+m\rho_1\dot\theta_1^2
 +h\rho_1\rho_2+h\rho_1^2\dot\theta_1+ 2 \dot\theta_1\dot\theta_2 +2\mu_j\dot\rho_1 \dot\rho_2\right)\dint s \,,
\label{eq:E3}\\
E_4&=\frac{1}{2}\int_0^{2\pi}\biggl(2m\rho_1\dot\theta_3+2m\rho_3\dot\theta_1+ m\rho_2\dot\theta_1^2+2m\rho_2\dot\theta_2
+2m\rho_1\dot\theta_1\dot\theta_2+\frac{h}{2}\rho_2^2+h\rho_1\rho_3\nn\\
&\hspace{4.5em}
+2h\rho_1\rho_2\dot\theta_1+{\frac{h}{2}}\rho_1^2\dot\theta_1^2+h\rho_1^2\dot\theta_2
+2 \dot\theta_1 \dot\theta_3+\dot\theta_2^2
+\mu_j\dot\rho_2^2 +2\mu_j\dot\rho_1 \dot\rho_3-\sigma\dot\rho_1^2\biggr)\dint s \,.
\label{eq:E4}
\end{align}

\begin{lemma}[Energy expansion]\label{lem:expansion}
Let $j\in\mathbb{N}$, let $\bar u_1,\bar u_2$ be given by \eqref{u11}, \eqref{eq:u21} in Case 1, $j\ge 2$, or by \eqref{u12}, \eqref{eq:u22} in Case 2, $j=1$.
Let the constants $a_1$ in Case 1 or $b_1$ in Case 2 be chosen such that the third-order inhomogeneities \eqref{eq:inhom31}--\eqref{eq:inhom33}
satisfy the solvability conditions of Lemma \ref{lem:Case0order3} in Case 1 and Lemma \ref{lem:Case2order3} in Case 2. Let $\bar u_3$ be a corresponding
solution of the linearized problem \eqref{eq:ELone}--\eqref{eq:theta0_lin} with $\mu=\mu_j$. Then the coefficients in the energy expansion above satisfy
$E_1=E_2=E_3=0$ in both cases, as well as
\beq\label{eq:E4Case010}
\hspace{-1em}E_{4}=-\frac{2\pi j^4(2m^2-h)}{Z(m,h)}
\eeq
in Case 1 with the notation of Lemma \ref{lem:Case0order3}. In Case 2 we have
\[
E_4=\frac{2\pi (3h+2m^2)}{3h^3} \,.
\]
\end{lemma}

As expected, the sign of $E_4$ goes with criticality of the bifurcating branch. Stability is gained ($E_4<0$) along supercritical branches and lost ($E_4>0$)
along subcritical branches. In particular, this can be expected to decide the stability of the branch corresponding to the first bifurcation for decreasing $\mu$.

\section{Numerical continuation of bifurcation branches}\label{sec:numerics}

\subsection{Discretization}
\label{sec:discretization}

The Euler-Lagrange equations \eqref{eq:ELrho} and \eqref{eq:ELtheta} are discretized by finite differences as follows.
For $N\in\mathbb{N}$ we discretize the interval $[0,L]$ by introducing $\Delta s = L (N-1)^{-1}$ and $s_i = i \Delta s$, $0 \le i \le N-1$, which naturally leads to the (abuse of) notation
$\rho=({\rho}_i)_{i=1}^{N-1}$, $\theta=({\theta}_i)_{i=1}^{N-1}$ with
${\rho}_i = \rho(s_i)$, ${\theta}_i = \theta(s_i)$.
This can be thought of as considering a polygonal approximation of the curve $\gamma$, where $\theta_i$ is the angle of the $i$\textsuperscript{th} side and where $\rho_i$ is a piecewise constant approximation of $\rho$ on that side (\emph{i.e.} $\rho_i$ is \emph{not} associated to a vertex).

Using the notation $u = (\rho, \theta)$ and $\Lambda=(\lambda_x,\lambda_y,\lambda_M)$,
we propose the following natural finite differences approximation for \eqref{eq:ELrho} and \eqref{eq:ELtheta},
respectively:
\begin{align}
    EL_\rho(u, \Lambda) &= \veps \left(\frac{\rho_{i-1} - 2 \rho_i + \rho_{i+1}}{\Delta s^2}\right) - \frac{1}{2} \beta'(\rho_i) \left(\frac{\rho_{i+1} - \rho_{i-1}}{2 \Delta s}\right)^2 - \lambda_M = 0\,,
    \label{eq:EL rho discretized}
    \\
    EL_\theta(u, \Lambda) &= \frac{1}{\Delta s}\left(\beta\left(\frac{\rho_{i+1}+\rho_i}{2}\right)\left(\frac{\theta_{i+1} - \theta_{i}}{\Delta s}\right) - \beta\left(\frac{\rho_{i}+\rho_{i-1}}{2}\right) \left(\frac{\theta_{i} - \theta_{i-1}}{\Delta s}\right)\right) + \lambda_x \sin \theta_i - \lambda_y \cos \theta_i = 0\,, \label{eq:EL theta discretized}
\end{align}
for $0 \le i \le N-1$.

To remove the degree of freedom associated to solid rotations, we can set $\theta(0) = 0$ at the continuous level. This is reflected by the choice $\theta_0 = 0$ at the discrete level.
We also need to provide values for indices $i = {-1, N}$. Again by periodicity we set
$\rho_{-1} = \rho_{N-1}$, $\rho_{N} = \rho_0$, $\theta_{-1} = \theta_{N-1} - 2\pi$, and $\theta_{N} = \theta_0 + 2\pi$. Thus, we only consider \eqref{eq:EL rho discretized} for $0 \le i < N-1$ and \eqref{eq:EL theta discretized} for $0 < i < N-1$.

The mass and closedness constraints can be naturally approximated as
\begin{align*}
    C_M(u, \Lambda) &= \Delta s \sum_{i=0}^{N-1} \rho_i - M = 0\,,\\
      \begin{pmatrix} C_{x}\\  C_{y}\end{pmatrix}(u, \Lambda) &= \Delta s \sum_{i=0}^{N-1} \begin{pmatrix}\cos \theta_i \\ \sin \theta_i\end{pmatrix} = 0\,.
\end{align*}

We are left with a system of $2N + 1$ nonlinear equations which we propose to solve using a damped Newton method. If we assume $\bar{u}^k = \left(u^k, \Lambda^k\right)$ to be known, we look for $\bar{u}^{k+1}$ as a solution to
\begin{equation}
    J(\bar{u}^k)\, (\bar{u}^{k+1} - \bar{u}^k) = - \eta\, r(\bar{u}^k)\,,
    \label{eq:newton non linear system}
\end{equation}
where $r(\bar{u}^k) = \left(EL_\rho, EL_\theta, C_x, C_y, C_M\right)$,
$J$ is the Jacobian of $r$ with respect to $\bar{u}$, and
$\eta \le 1$ is the damping parameter with $\eta = 1$ corresponding to the standard Newton's method.

\subsubsection{Continuation of branches}

To follow numerically the bifurcation branches, one can pick some $\veps$ close to the critical value $\mu_j$ and take as initial value a perturbation of the trivial solution (corresponding to the circle with homogeneous $\rho$).
The position of $\veps$ relative to the critical value and the amplitude of the bifurcation are given precisely by the results of Section \ref{sec:bifurcation}.
Solving \eqref{eq:newton non linear system} yields a numerical
approximation of a critical point, which can be used as initial condition for neighbouring values of $\veps$. By iterating this process, one can move along the branch, provided that
\begin{enumerate}
    \item the branch is locally smooth (for example, this is not the case when $\rho$ hits zeros of $\beta$, where one could expect the branch to terminate),
    \item the features of the solution can be resolved by the discretization with the chosen value of $N$.
\end{enumerate}

\subsection{Choice of parameters}
In what follows we will consider a number of different situations, depending on the choice of parameters $(m, h)$ for the function $\beta$, which will be of the form \eqref{eq:beta-ass} with $\beta_0=1$, namely
\begin{equation*}
    \beta(\rho) =  1 + m\left(\rho - \rho_0\right) + \frac{h}{2} \left(\rho-\rho_0\right)^2\,.
\end{equation*}
As before we will take $M = L = 2\pi$, so that $\rho_0 = 1$. We consider six sets of parameters:
\begin{enumerate}[(i)]
    \item $(m,h) = (1,1.85)$ corresponding to Case 1 with $\sigma = -1$ (subcritical bifurcation),
    \item $(m,h) = (1,1)$ corresponding to Case 1 with $\sigma = 1$ (supercritical bifurcation),
    \item $(m,h) = (1,1/4)$ corresponding to Case 1 with $\sigma = -1$ (subcritical bifurcation),
    \item $(m,h) = (1, -1/2)$ corresponding to Case 1 with $\sigma = -1$ (subcritical bifurcation) and supercritical Case 2,
    \item $(m,h) = (1, -2)$ corresponding to Case 1 with $\sigma = -1$ (subcritical bifurcation) and subcritical Case 2,
    \item $(m,h) = (0, -1)$, which is similar to (v) in the special choice $m=0$. At first order, $\gamma$ should remain a circle, including for Case 1, as the correction coefficient $\theta_1 \equiv 0$ for $m=0$, see \eqref{u11}.
\end{enumerate}

For the definition of the different cases, we refer to Proposition \ref{prop:lin}. The parameters in (i) -- (vi) are represented in Figure \ref{fig:numerics bifurcation diagram}. The corresponding results are presented in Figures~\ref{fig:numerics bifurcation branches} and~\ref{fig:numerics shapes}.

\begin{figure}[htb!]
    \begin{center}
        \includegraphics{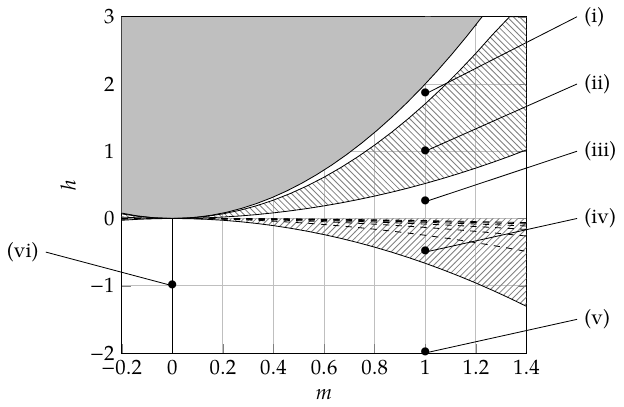}
    \end{center}
    \caption{The different sets of model parameters $(m,h)$ represented on the parameter space. The gray region corresponds to parameters which have no critical points except the trivial solution. The crosshatched region corresponds to supercritical bifurcations (Case 1 for $h>0$, Case 2 for $h<0$), and the plain white region to subcritical bifurcations. The dashed parabolas indicate where Case 3 occurs, for $j$ up to $8$.}
    \label{fig:numerics bifurcation diagram}
\end{figure}

\begin{figure}[htb!]
    \begin{center}
        \includegraphics{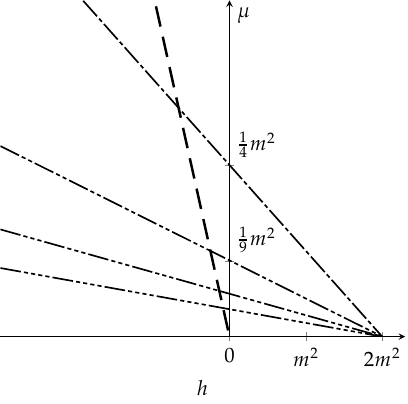}


    \end{center}
    \caption{Critical values of $\veps$ for Case 1 (thin) and Case 2 (bold). The intersections correspond to (the degenerate) Case 3 which is not studied in this paper. The dashes indicate the value of $j$: ---~--- for $j=1$, ---~-~--- for $j=2$, etc.}
\end{figure}

\subsection{Results}
\label{sec:numerics results}

The method described above (Section~\ref{sec:discretization}) was implemented in \texttt{Julia} \cite{bezanson_julia_2017}.  Figure~\ref{fig:numerics bifurcation branches} presents the bifurcating branches both in terms of the amplitude of the density $\rho$ and in terms of the energy $E_\veps$. It offers a partial confirmation of the results of Section~\ref{sec:bifurcation} in that:
\begin{itemize}
    \item For Case (ii), $j \ge 2$ and (iv), $j = 1$, the bifurcation appears supercritical, i.e., the branch bifurcates to the left of the critical $\veps$. Additionally, the energy decreases close to the trivial state. These branches offer critical points of $E_\veps$ which are candidates to be global minimizers.
    \item For all other cases, the bifurcation is subcritical, i.e., the branch bifurcates to the right of the critical $\veps$, and the energy initially increases as one gets further from the trivial state.
\end{itemize}
Interestingly, Cases (i) and (iv) feature turning points, where the derivative of $E_\veps$ along the branch seems to change sign. In Case (i) it becomes negative, leading to critical points of lower energy with respect to the trivial state, and potentially global minimizers. This fact precludes uniqueness of minimizers of $E_\veps$ in general.

We were able to track an additional branch in Cases (i) to (iii), which seem to bifurcate from the $j=2$ branch. No analytical results are available at this point, but we can make the following observations. The corresponding shapes, presented in gray in Figure~\ref{fig:numerics shapes}, look like the ones obtained for $j=1$ in Cases (iv) to (vi). This justifies the placement in the first column, although $j$ has no meaning for this branch. In Cases (i) and (iii), it bifurcates from the $j=2$ branch with decreasing energy for the choice of parameter considered. Case (ii) is a bit different, in that the bifurcation leads to critical points of higher energy, although the branch features a turning point, after which $E_\veps$ starts decreasing and eventually becomes smaller than for the $j=2$ branch, for a given value of~$\veps$.

Other features of the critical points further along the branch can be seen in Figure~\ref{fig:numerics shapes}. For Cases (i) to (v) and $j>1$, one can identify the value of $j$ with the number of flatter sections in each closed curve. These correspond to higher values of $\rho$, which agree with the fact that for all choices of parameters presented here, $m\ge0$. This can be roughly thought as higher values of $\rho$ penalizing higher values of the curvature $\dot\theta$. For $j=1$ or in Case (vi), the situation is different, since the first-order correction $\theta_1 \equiv 0$. If one goes further in the expansion, one case expect that the next order correction $\theta_2$ have the form $\cos(2js)$, that is half the period of $\rho_1$. This could explain that in these cases, $\rho$ seems to have $j$ maxima when $\dot\theta$ has $2j$.

Additionally, \emph{far} from the bifurcation point and after potential turning points, one can distinguish Cases (i) and (ii) from Cases (iii) to (vi). For the former, $\veps$ decreases along the branch, $E_\veps$ decreases and $\rho$ seems to concentrate on flat sections. For the latter, the situation is the opposite: $\veps$ increases along the branch, $E_\veps$ increases and $\rho$ stays rather smooth.

\begin{remark}
    In Proposition~\ref{thm:trivial_unique_large_epsilon} it is stated that for $\beta$ bounded away from $0$, only the trivial state $(\rho_0, \theta_0)$ is a minimizer of $E_\veps$.
    The branches in Figure~\ref{fig:numerics bifurcation branches} which seem to continue far to large values of $\veps$ have an energy clearly larger than $\pi = E_\veps(\rho_0, \theta_0)$.
    We also recall that for the results presented here, the choice of $\beta$ is quadratic, and thus not bounded away from $0$. There is then no contradiction of our analysis.
\end{remark}

A systematic study of the stability in terms of the energy would be interesting, although probably necessarily limited to numerics, as it would help identifying local minimizers. Such an investigation is however out of the scope of this paper.

\begin{figure}[h!]
\makeatletter
\def\tikz@auto@anchor{%
    \pgfmathtruncatemacro\angle{atan2(\pgf@x,\pgf@y)-90}
    \edef\tikz@anchor{\angle}%
}
\makeatother

    \begin{center}
\vspace{-2em}
        \includegraphics{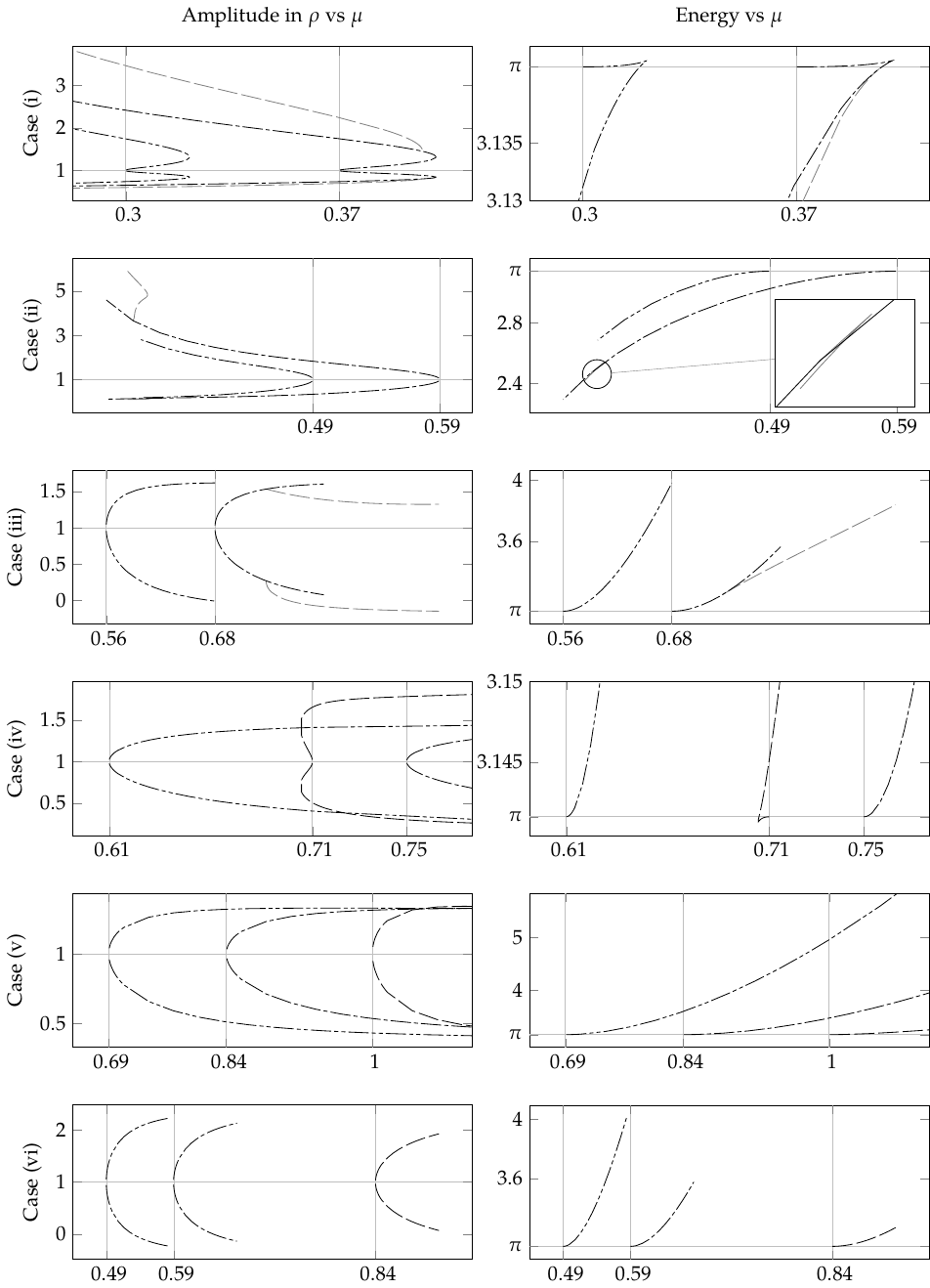}
    \end{center}
    \caption{Numerical results for Cases (i) to (vi), in columns, for $j=1,2,3$. The first column shows the amplitude in $\rho$, where the lower and greatest values of $\rho$ are represented. The horizontal gray line corresponds to the trivial solution, for which $\rho \equiv 1$.
        The second column shows the energy $E_\veps$, with the horizontal gray line again corresponding to the trivial solution, for which $E_\veps = \pi$. The dashes indicate the value of $j$ for each branch: ---~--- for $j=1$ (absent in (i) to (iii)), ---~-~--- for $j=2$, ---~-~-~--- for $j=3$. The gray vertical lines indicate the theoretical critical values for $\veps$. In Cases (i) to (iii), the secondary bifurcation branch is plotted in gray.
As detailed in~\eqref{eq:perturbationansatz}, at a supercritical (resp.\ subcritical) bifurcation point, the branch will appear for values of $\veps$ greater (resp.\ lower) than the critical value.}
    \label{fig:numerics bifurcation branches}
\end{figure}

\begin{figure}[h!]
    \begin{center}
        \begin{tabularx}{\textwidth}{cccc}
            & Case 2 & Case 1 & Case 1 \\
            & $j = 1$ & $j = 2$ & $j = 3$
            \\
            \raisebox{0.1\textwidth}{Case (i)} &
            \includegraphics[trim=200 10 200 10,clip,width=0.25\textwidth]{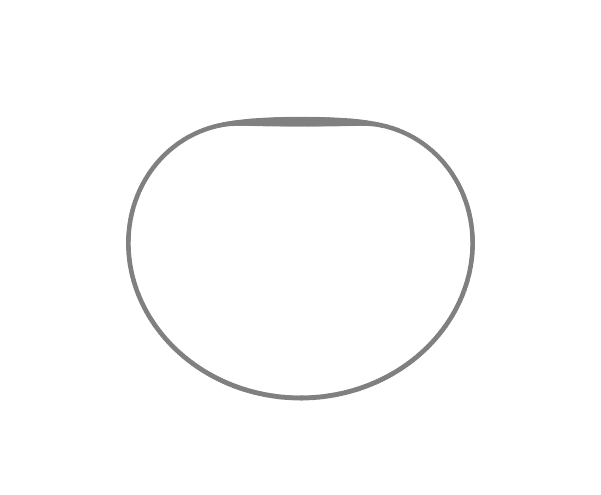}
                                                 &
            \includegraphics[trim=200 10 200 10,clip,width=0.25\textwidth]{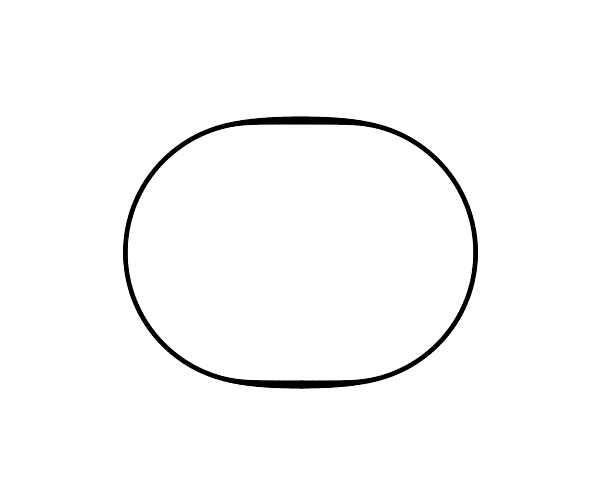}
                                               &
            \includegraphics[trim=200 10 200 10,clip,width=0.25\textwidth]{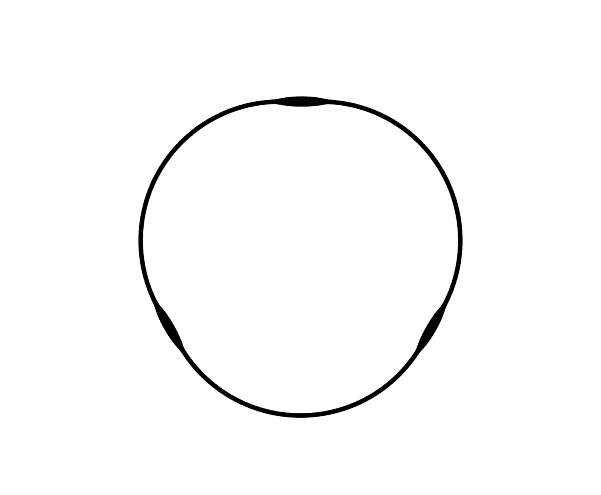}
        \\
            \raisebox{0.1\textwidth}{Case (ii)} &
            \includegraphics[trim=200 10 200 10,clip,width=0.25\textwidth]{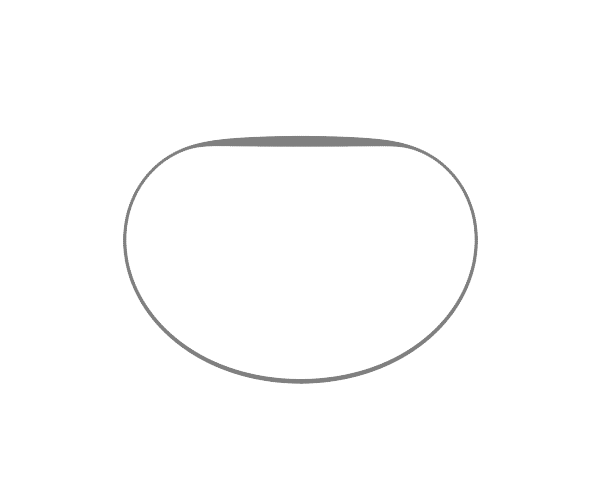}
                                                 &
            {\includegraphics[trim=200 10 200 10,clip,width=0.25\textwidth]{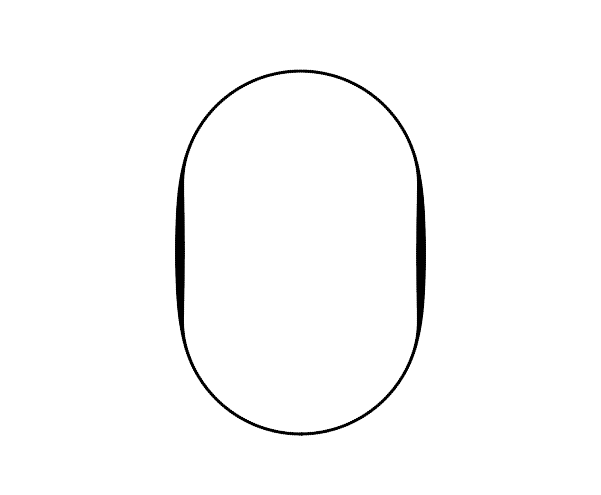}}
                                               &
           {\includegraphics[trim=200 10 200 10,clip,width=0.25\textwidth]{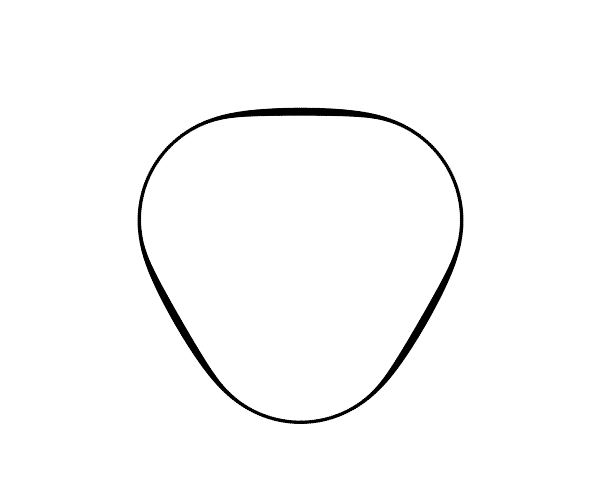}}
        \\
            \raisebox{0.1\textwidth}{Case (iii)} &
            \rotatebox{-3}{\includegraphics[trim=200 10 200 10,clip,width=0.25\textwidth]{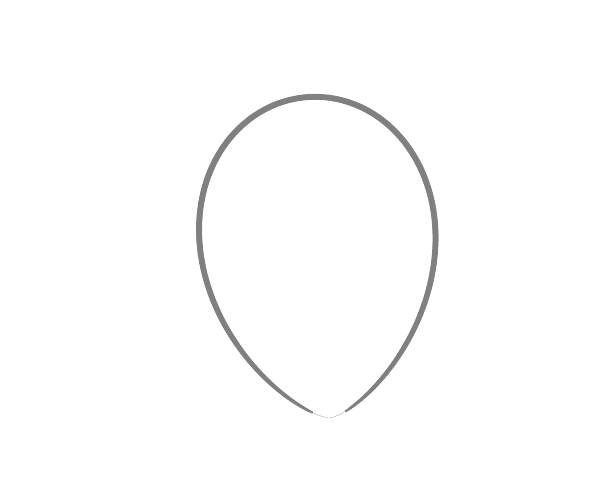}}
                                                 &
            \includegraphics[trim=200 10 200 10,clip,width=0.25\textwidth]{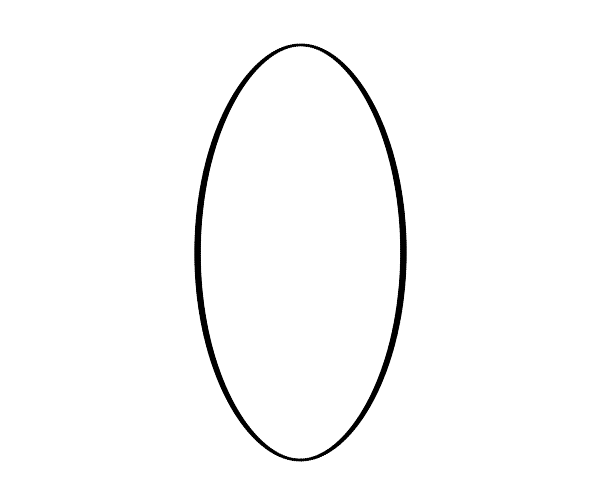}
                                               &
            \includegraphics[trim=200 10 200 10,clip,width=0.25\textwidth]{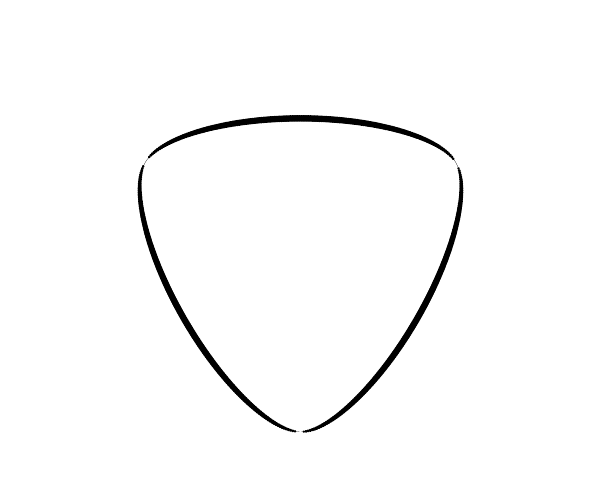}
        \\
            \raisebox{0.1\textwidth}{Case (iv)} &
            {\includegraphics[trim=200 10 200 10,clip,width=0.25\textwidth]{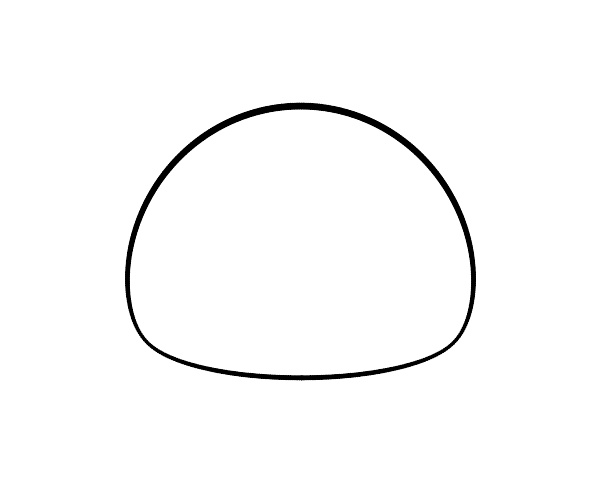}}
                                               &
            \includegraphics[trim=200 10 200 10,clip,width=0.25\textwidth]{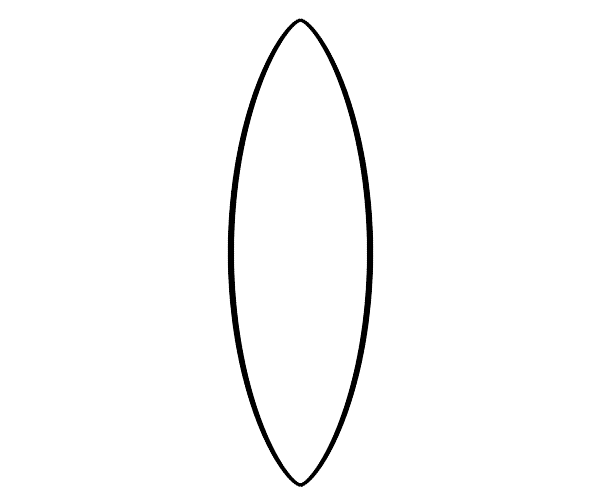}
                                               &
            \includegraphics[trim=200 10 200 10,clip,width=0.25\textwidth]{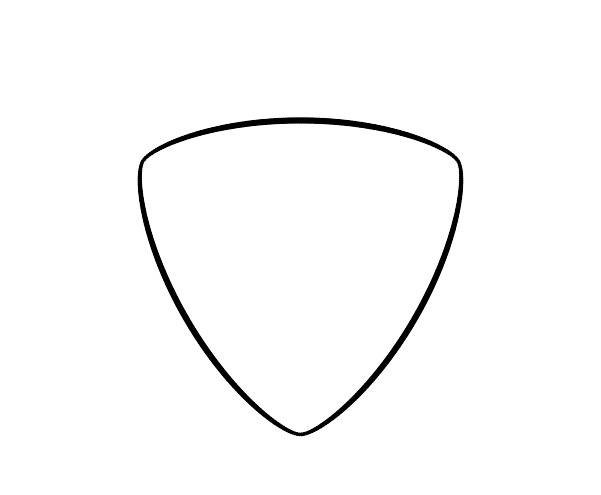}
        \\
            \raisebox{0.1\textwidth}{Case (v)} &
            \includegraphics[trim=200 10 200 10,clip,width=0.25\textwidth]{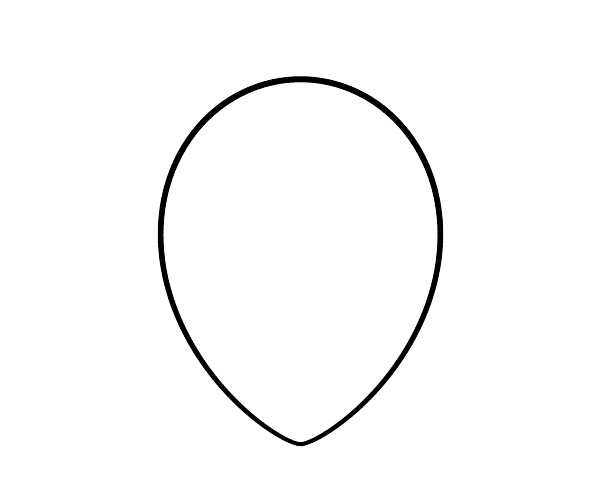}
                                               &
            \includegraphics[trim=200 10 200 10,clip,width=0.25\textwidth]{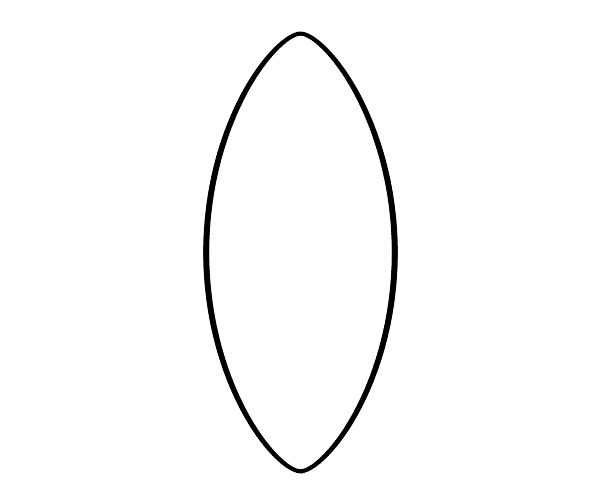}
                                               &
            \includegraphics[trim=200 10 200 10,clip,width=0.25\textwidth]{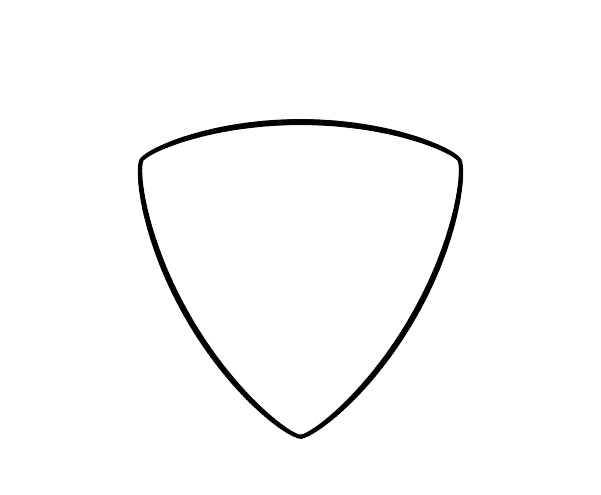}
        \\
            \raisebox{0.1\textwidth}{Case (vi)} &
        \includegraphics[trim=200 10 200 10,clip,width=0.25\textwidth]{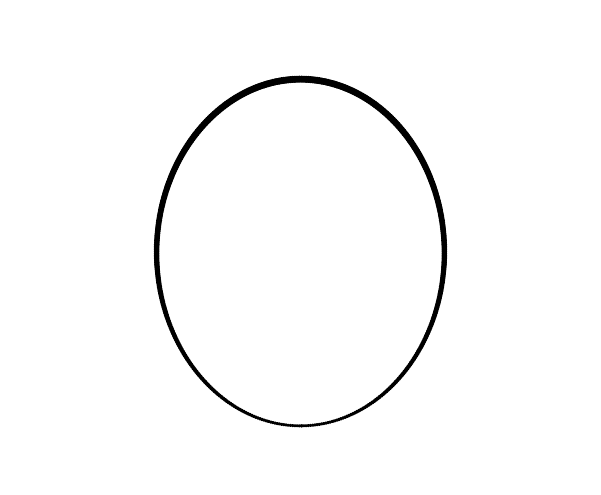}
                                                &
        \includegraphics[trim=200 10 200 10,clip,width=0.25\textwidth]{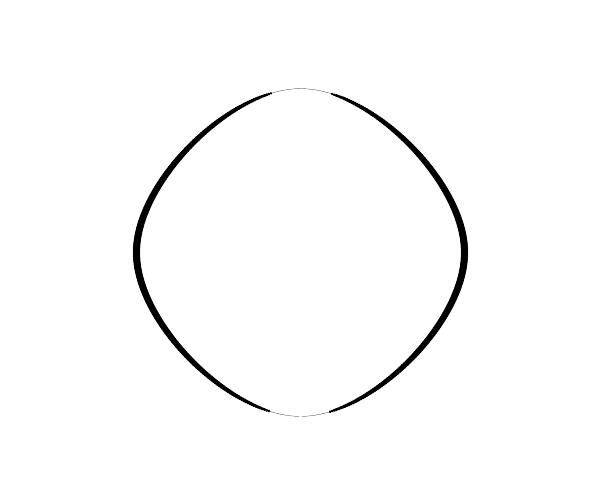}
                                               &
        \includegraphics[trim=200 10 200 10,clip,width=0.25\textwidth]{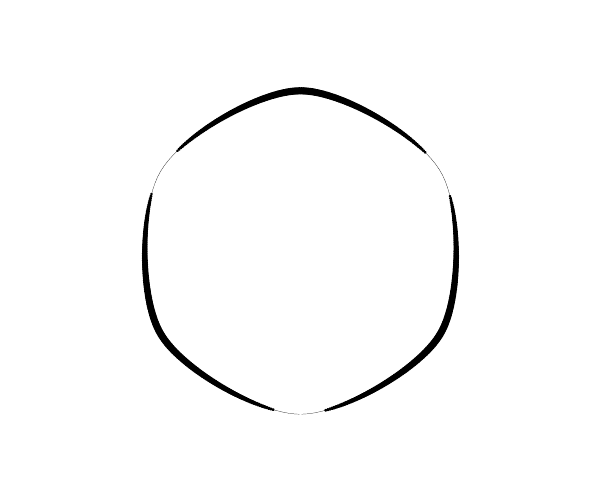}
\end{tabularx}
    \end{center}
    \caption{The shapes corresponding to each case, with $j = 1,2,3$ increasing with each column.
        These correspond to the last point computed on the branches shown in Figure~\ref{fig:numerics bifurcation branches}.
        In Cases (i) to (iii), the shapes in the first column are in gray, as they do not correspond to branches bifurcating from the trivial state, and $j$ is not defined in this case. They are placed in the first column due to their resemblance to shapes obtained for $h < 0$, in Cases (iv) to (vi).
        Thicker lines denote larger values of $\rho$.
    }
    \label{fig:numerics shapes}
\end{figure}

\clearpage

\section{Conclusion}\label{sec:conclusion}
To describe elastic curves in the plane, we introduced a regularized Canham-Helfrich type functional which includes a density-modulated stiffness $\beta$.
We proved that the associated minimization problem has a solution if the regularization parameter is positive. If not, the problem has no solution in general.
Conditions on the first derivatives of $\beta$ were derived so that the problem has non trivial solutions.
In this case, a bifurcation analysis around the trivial solution was performed, the regularization parameter playing the role of the bifurcation parameter.
A family of both subcritical and supercritical pitchfork bifurcations were found, depending on the choice of $\beta$.
This contrasts with the classical elastic curves, which display supercritical bifurcations only.
An expansion of the energy confirmed that subcritical (resp.\ supercritical) solutions correspond to a gain (resp.\ a loss) of energy compared to the trivial state.
This analysis was completed by numerical continuation of the bifurcating branches, which confirmed the theoretical findings.
Secondary bifurcations and turning points ---found numerically--- testify of the intricate mathematical structure of the model.
In particular no uniqueness should be expected for the minimization problem, except for large regularization.

\section*{Acknowledgements}
This work has been supported by the Austrian Science Fund (FWF)
projects F\,65, W\,1245, P\,32788, by the Vienna Science and Technology Fund (WWTF)
through Project MA14-009, and by the Austrian Academy of Sciences via the New Frontier's grant NST 0001.

\addcontentsline{toc}{section}{Acknowledgements}


\bibliography{BJSS.2020}
\bibliographystyle{abbrv}

\end{document}